\documentclass[a4paper,12pt]{article}

\usepackage{dsfont, amsmath, amsthm, amsfonts, amssymb,accents}
\usepackage{graphicx}
\usepackage{srcltx}

\theoremstyle{plain}

 \newtheorem{thm}{Theorem}[section]
 
 \newtheorem{lem}[thm]{Lemma}

\theoremstyle{definition}

\theoremstyle{rem}
 \newtheorem{rem}[thm]{Remark}
 
 \numberwithin{equation}{section}

\def\tht{\theta}
\def\Om{\Omega}

\def\e{\varepsilon}
\def\g{\gamma}
\def\G{\Gamma}
\def\l{\lambda}
\def\p{\partial}
\def\D{\Delta}

\def\k{\varkappa}
\def\E{\mbox{\rm e}}
\def\a{\alpha}
\def\b{\beta}

\def\d{\delta}
\def\L{\Lambda}
\def\z{\zeta}
\def\vs{\varsigma}
\def\r{\rho}

\def\Odr{\mathcal{O}}
\def\H{W_2}

\def\Ho{\mathring{W}_2}
\def\Hoper{\mathring{W}_{2,per}}

\def\di{\,\mathrm{d}}

\def\I{\mathrm{I}}
\def\iu{\mathrm{i}}

\def\Hpe{\mathring{\mathcal{H}}_\e}

\def\hpe{\mathring{\mathfrak{h}}_\e}

\def\Gp{\mathring{\G}}
\def\gp{\mathring{\g}}

\def\po{\mathring{\psi}}
\def\Po{\mathring{\Psi}}
\def\Pho{\mathring{\Phi}}
\def\pho{\mathring{\phi}}

 \DeclareMathOperator{\RE}{Re}
 \DeclareMathOperator{\spec}{\sigma}

\DeclareMathOperator{\essspec}{\sigma_{e}}

\DeclareMathOperator{\sgn}{sgn}

\begin{document}
\allowdisplaybreaks

\begin{center}

\Large{\textbf{On a waveguide with frequently alternating boundary conditions:
homogenized Neumann condition}}

\bigskip

\large{Denis Borisov$^1$, Renata Bunoiu$^2$, and Giuseppe Cardone$^3$}

\end{center}

\begin{quote}
{\small {\em 1) Bashkir State Pedagogical University, October
Revolution St.~3a,
\\
\phantom{1) } 450000 Ufa, Russia, e-mail: \texttt{borisovdi@yandex.ru}
\\
2) LMAM, UMR 7122, Universit\'e de Metz et CNRS Ile du Saulcy,
\\
\phantom{1) } F-57045 METZ Cedex 1, France, e-mail: \texttt{bunoiu@math.univ-metz.fr}
\\
3) University of Sannio, Department of Engineering, Corso Garibaldi, 107,
\\
\phantom{1) } 82100 Benevento, Italy, e-mail:
\texttt{giuseppe.cardone@unisannio.it}}
}
\end{quote}


\def\thefootnote{}

\footnote{This work was partially done during the visit of D.B. to the University of
Sannio (Italy) and of G.C. to LMAM of University Paul Verlaine of Metz (France).
They are grateful for the warm hospitality extended to them.
D.B. was partially supported by RFBR (09-01-00530),
by the grants of the President of Russia for young
scientists-doctors of sciences (MD-453.2010.1) and for Leading
Scientific School (NSh-6249.2010.1),  by Federal Task Program
``Research and educational professional community of innovation
Russia'' (contract 02.740.11.0612), and by the project ``Progetto ISA: Attivit\`a di
Internazionalizzazione dell'Universit\`a degli Studi del Sannio''.}

\date{}

\begin{abstract}
We consider a waveguide modeled by the Laplacian in a straight
planar strip. The Dirichlet boundary condition is taken on the upper
boundary, while on the lower boundary we impose periodically
alternating Dirichlet and Neumann condition assuming the period of
alternation to be small. We study the case when the homogenization
gives the Neumann condition instead of the alternating ones. We
establish the uniform resolvent convergence and the estimates for
the rate of convergence. It is shown that the rate of the
convergence can be improved by employing a special boundary
corrector. Other results are the uniform resolvent convergence for
the operator on the cell of periodicity obtained by the
Floquet-Bloch decomposition, the two-terms asymptotics for the band
functions, and the complete asymptotic expansion for the bottom of
the spectrum with an exponentially small error term.
\end{abstract}

\section{Introduction}

During last decades, models of quantum waveguides attracted much attention by both physicists and mathematicians. It was motivated by many interesting mathematical phenomena of these models and also by the progress in the semiconductor physics, where they have important applications. Much efforts were exerted to study influence of various perturbations on the spectral properties of the waveguides. One of such perturbations is a finite number of openings coupling two lateral waveguides (see, for instance, \cite{JPA-B}, \cite{B-MS06}, \cite{BEG-02}, \cite{BGRS}, \cite{DK}, \cite{ESTV}, \cite{EV2}). Such openings are usually called ``windows''. If the coupled waveguides are symmetric, one can replace them by a single waveguide with the opening(s) modeled by the change of boundary condition (see \cite{BEG-02}, \cite{BGRS}, \cite{DK}). The main phenomenon studied in \cite{JPA-B}, \cite{B-MS06}, \cite{BEG-02}, \cite{BGRS}, \cite{DK}, \cite{ESTV}, \cite{EV2} is the appearance of new eigenvalues below the essential spectrum, which is stable w.r.t.
windows.

A close model was suggested in \cite{BC}, where the number of openings was infinite. The waveguide was modeled by a straight planar strip, where the Dirichlet Laplacian was considered. On the upper boundary the Dirichlet condition was imposed. On the lower boundary the Neumann condition was settled on a periodic set, while on the remaining part of the boundary the Dirichlet condition is involved. In other words, on the lower boundary one had the alternating boundary conditions. The main assumption was the smallness of the sizes of Dirichlet and Neumann parts on the lower  boundary. They 
were described by two parameters: the first one, $\e$, was supposed to be small, while the other, $\eta=\eta(\e)$, could be either bounded or small.

The main difference between the models studied in \cite{BC} and in \cite{JPA-B}, \cite{B-MS06}, \cite{BEG-02}, \cite{BGRS}, \cite{DK},
\cite{ESTV}, \cite{EV2} is the influence of the perturbation on the
spectral properties: while in the latter papers the essential
spectrum remained unchanged and discrete eigenvalues appeared below
its bottom, in \cite{BC} the spectrum was purely essential and had
band structure. Moreover, it depended on the perturbation and, for
example, the bottom of the spectrum moved as $\e\to+0$. Assuming
that
\begin{equation}\label{0.1}
\e\ln\eta(\e)\to-0 \quad\text{ as }\e\to+0,
\end{equation}
it was shown in \cite{BC} that the homogenized operator is the
Laplacian with the previous boundary condition on the upper
boundary, while the alternation on the lower boundary should be
replaced by the Dirichlet one. More precisely, it was shown that the
uniform resolvent convergence for the perturbed operator holds true
and the rate of convergence was estimated. Other main results were
the two-terms asymptotics for first band functions of the perturbed
operator and the complete two-parametric asymptotic expansion for
the bottom of the spectrum.

In the present paper we consider a different case: we assume that
the homogenized operator has the Neumann condition on the lower
boundary, which is guaranteed by the condition
\begin{equation}\label{1.5}
\e\ln\eta(\e)\to -\infty \quad \text{ as }\e\to+0.
\end{equation}
We observe that this condition is not new, and it was known before
that it implied the homogenized Neumann boundary condition for the
similar problems in bounded domains, see \cite{Gzh01}, \cite{Ch},
\cite{Ch2}, \cite{ChD1}, \cite{ChD2}, \cite{Fr}.

We obtain
the uniform resolvent convergence for the perturbed operator and we
estimate the rate of convergence. We also obtain similar
convergence for the operator appearing on the cell of periodicity
after Floquet decomposition and provide two-terms asymptotics for
the first band function. The last main result is the complete
asymptotic expansion for the bottom of the spectrum.

 Similar results were obtained \cite{BC} under the assumption
(\ref{0.1}), and now we want to underline the main differences. We
first observe that in \cite{BC} the estimate of the rate of
convergence for the perturbed resolvent was obtained for the
difference of the resolvents of the perturbed and homogenized
operator and this difference was considered as an operator from
$L_2$ into $\H^1$. In our case, in order to have a similar good
estimate, we have to consider the difference not with the resolvent
of the homogenized operator, but with that of an additional operator
depending in boundary condition on an additional parameter
\begin{equation}\label{1.6}
\mu=\mu(\e):=-\frac{1}{\e\ln\eta(\e)}\to+0 \quad \text{ as }\e\to+0.
\end{equation}
Moreover, we also have to use a special boundary \emph{corrector}, see Theorem~\ref{th1.1}. Omitting the corrector and estimating the difference of the same resolvents as an operator in $L_2$, we can still preserve the mentioned good estimate. Omitting the corrector or replacing the additional operator mentioned above by the homogenized one, one worsens the rate of convergence. At the same time, this rate can be improved partially by considering the difference of the resolvents as an operator in $L_2$. Such situation was known to happen in the case of the operators with the fast oscillating coefficients (see \cite{BS2}, \cite{BS5}, \cite{BoAA08}, \cite{PT}, \cite{Pas}, \cite{Su2}, \cite{Su1}, \cite{SuKh}, \cite{Zh3}, \cite{Zh4} and the references therein for further results). From this point of view the results of the present paper are closer to the cited paper in contrast to the results of \cite{BC} and \cite[Ch. I\!I\!I, Sec. 4.1]{OIS}.

One more difference to \cite{BC} is the asymptotics for the band functions and the bottom of the essential spectrum. The second term in the asymptotics for the band functions is not a constant, but a holomorphic in $\mu$ function. In fact, it is a series in $\mu$ and this is why the mentioned two-terms asymptotics can be regarded as the asymptotics with more terms, see (\ref{1.15}). Even more interesting situation occurs in the asymptotics for the bottom of the spectrum. Here the asymptotics contains just one first term, but the error estimate is \emph{exponential}. The leading term depends on $\e$ and $\mu$ \emph{holomorphically} and can be represented as the series in $\e$ with the holomorphic in $\mu$ coefficients. For the bounded domains the complete asymptotic expansions for the eigenvalues in the case of the homogenized Neumann problem were constructed in \cite{AsAn}, \cite{GDu99}. These asymptotics were power in $\e$ \cite{GDu99} with the holomorphic in $\mu$ coefficients \cite{AsAn}. At the same time, the error terms were powers in $\e$ and the convergence of these asymptotic series was not proved. In our case the first term in the asymptotics for the bottom of the essential spectrum is the sum of the asymptotic series analogous to those in \cite{AsAn}, \cite{GDu99}. In other words, we succeeded to prove that in our case this series converges, is holomorphic in $\e$ and $\mu$ and gives the exponentially small error term that for singularly perturbed problems in homogenization is regarded as a strong result.

Eventually, we point out that the technique we use is different: in addition to the boundary layer method \cite{VL} used also in \cite{BC}, here we also have to employ the method of matching of the asymptotic expansions \cite{Il}. Such combination was borrowed from \cite{AsAn}, \cite{GD98}, \cite{Gzh01}, \cite{GDu99}. We use this combination to construct the aforementioned corrector to obtain the uniform resolvent convergence. Similar correctors were also constructed in \cite{Ch}, \cite{Fr}, \cite{Gzh01}, but to obtain either weak or strong resolvent convergence. We also employ the same corrector in the combination of the technique developed in \cite{FS} for the analysis of the uniform resolvent convergence for thin domains.

In conclusion, we describe briefly the structure of
the paper. In the next section we formulate precisely the problem
and give the main results. The third section is devoted to the
study of the uniform resolvent convergence. In the fourth section we
make the similar study for the operator appearing after the Floquet
decomposition, and we also establish two-terms asymptotics for the
first band functions. In the last, fifth section we construct the complete
asymptotic expansion for the bottom of the spectrum.

\section{Formulation of the problem and the main results}

Let $x=(x_1,x_2)$ be Cartesian coordinates in $\mathds{R}^2$, and
$\Om:=\{x: 0<x_2<\pi\}$ be a straight strip of width $\pi$. By $\e$
we denote a small positive parameter, and $\eta=\eta(\e)$ is a
function satisfying the estimate
\begin{equation*}
0<\eta(\e)<\frac{\pi}{2}.
\end{equation*}
We indicate by $\G_+$ and $\G_-$ the upper and lower boundary of
$\Om$, and we partition $\G_-$ into two subsets (cf.
fig.~\ref{fig1}),
\begin{equation*}
\g_\e:=\{x: |x_1-\e\pi j|<\e\eta,\, x_2=0,\, j\in \mathds{Z}\},
\quad \G_\e:=\G_-\setminus\overline{\g_\e}.
\end{equation*}

The main object of our study is the Laplacian in $L_2(\Om)$ subject
to the Dirichlet boundary condition on $\G_+\cup\g_\e$ and to the
Neumann one on $\G_\e$. We introduce this operator as the
non-negative self-adjoint one in $L_2(\Om)$ associated with the
sesquilinear form
\begin{equation*}
\mathfrak{h}_\e[u,v]:=(\nabla u,\nabla v)_{L_2(\Om)}\quad
\text{on}\quad \Ho^1(\Om,\G_+\cup\g_\e),
\end{equation*}
where $\Ho^1(Q,S)$ indicates the subset of the functions in
$\H^1(Q)$ having zero trace on the curve $S$. We denote the
described operator as $\mathcal{H}_\e$. The aim of this paper is to
study the asymptotic behavior of the resolvent and the spectrum of
$\mathcal{H}_\e$ as $\e\to+0$.

Let $\mathcal{H}^{(\mu)}$ be the non-negative self-adjoint operator
in $L_2(\Om)$ associated with the sesquilinear form
\begin{equation*}
\mathfrak{h}^{(\mu)}[u,v]:=(\nabla u,\nabla
v)_{L_2(\Om)}+\mu(u,v)_{L_2(\p\Om)}\quad \text{on}\quad
\Ho^1(\Om,\G_+),
\end{equation*}
where $\mu\geqslant 0$ is a constant. Reproducing the arguments of
\cite[Sec. 3]{IEOP}, one can show that the domain of
$\mathcal{H}^{(\mu)}$ consists of the functions in $\H^2(\Om)$
satisfying the boundary condition
\begin{equation}\label{1.4}
\frac{\p u}{\p x_2}-\mu u=0\quad \text{on} \quad \G_-,\qquad u=0
\quad \text{on} \quad \G_+,
\end{equation}
and
\begin{equation}\label{1.11}
\mathcal{H}^{(\mu)}u=-\D u.
\end{equation}
By $\|\cdot\|_{L_2(\Om)\to L_2(\Om)}$ and
$\|\cdot\|_{L_2(\Om)\to\H^1(\Om)} $ we denote the norm of an
operator acting from $L_2(\Om)$ into $L_2(\Om)$ and into
$\H^1(\Om)$, respectively.

Our first main result describes the uniform resolvent convergence
for $\mathcal{H}_\e$.

\begin{figure}[t]
\begin{center}
\includegraphics[width=11.7 true cm, height=2.416 true cm]{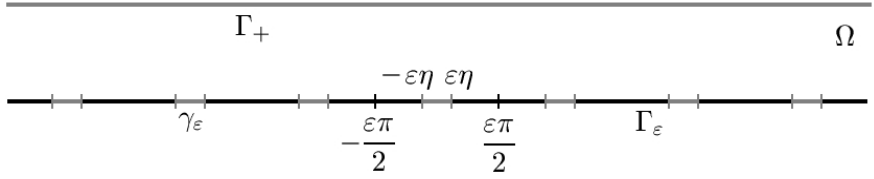}
\caption{Waveguide with frequently alternating boundary conditions}
\label{fig1}
\end{center}
\end{figure}

\begin{thm}\label{th1.1}
Suppose (\ref{1.5}). Then
\begin{align}
&\|(\mathcal{H}_\e-\iu)^{-1}-(\mathcal{H}^{(\mu)}-\iu)^{-1}
\|_{L_2(\Om)\to L_2(\Om)}\leqslant C \e\mu|\ln\e\mu|,\label{1.10}
\\
& \|(\mathcal{H}_\e-\iu)^{-1}-(\mathcal{H}^{(0)}-\iu)^{-1}
\|_{L_2(\Om)\to\H^1(\Om)}\leqslant C \mu^{1/2}, \label{1.8}
\\
& \|(\mathcal{H}_\e-\iu)^{-1}-(\mathcal{H}^{(0)}-\iu)^{-1}
\|_{L_2(\Om)\to L_2(\Om)}\leqslant C \mu, \label{1.8a}
\end{align}
where the constants $C$ are independent of $\e$ and $\mu$, 
and $\mu=\mu(\e)$ was defined in (\ref{1.6}). There
exists a corrector $W=W(x,\e,\mu)$ defined explicitly by
(\ref{2.19}) such that
\begin{equation}\label{1.9}
\|(\mathcal{H}_\e-\iu)^{-1}-(1+W)(\mathcal{H}^{(\mu)}-\iu)^{-1}
\|_{L_2(\Om)\to\H^1(\Om)}\leqslant C \e\mu|\ln\e\mu|,
\end{equation}
where the constant $C$ is independent of $\e$ and $\mu$.
\end{thm}

The spectrum of the operator $\mathcal{H}^{(0)}$ is purely essential
and coincides with $\left[\frac{1}{4},+\infty\right)$. By [RS1, Ch.
VIII, Sec. 7, Ths. VIII.23, VIII.24] and Theorem~\ref{th1.1} we have

\begin{thm}\label{th1.2}
The spectrum of $\mathcal{H}_\e$ converges to that of
$\mathcal{H}^{(0)}$. Namely, if $\l\not\in
\left[\frac{1}{4},+\infty\right)$, then
$\l\not\in\spec(\mathcal{H}_\e)$ for $\e$ small enough. If
$\l\in\left[\frac{1}{4},+\infty\right)$, then there exists
$\l_\e\in\spec(\mathcal{H}_\e)$ so that $\l_\e\to\l$ as $\e\to+0$.
The convergence of the spectral projectors associated with
$\mathcal{H}_\e$ and $\mathcal{H}^{(0)}$
\begin{equation*}
\|\mathcal{P}_{(a,b)}(\mathcal{H}_\e)-
\mathcal{P}_{(a,b)}(\mathcal{H}^{(0)})\|\to0,\quad \e\to0,
\end{equation*}
is valid for $a<b$.
\end{thm}

The operator $\mathcal{H}_\e$ is periodic since the sets $\g_\e$ and
$\G_\e$ are periodic, and we employ the Floquet decomposition to
study its spectrum. We denote
\begin{gather*}
\Om_\e:=\left\{x: |x_1|<\frac{\e\pi}{2}, \ 0<x_2<\pi\right\},
\\
\gp_\e:=\p\Om_\e\cap\g_\e,\quad\Gp_\e:=\p\Om_\e\cap\G_\e,\quad
\Gp_\pm:=\p\Om_\e\cap\G_\pm.
\end{gather*}
By $\Hpe(\tau)$ we indicate the self-adjoint non-negative operator
in $L_2(\Om_\e)$ associated with the sesquilinear form
\begin{equation*}
\hpe(\tau)[u,v]:=\left( \left(\iu\frac{\p}{\p
x_1}-\frac{\tau}{\e}\right)u,\left(\iu\frac{\p}{\p
x_1}-\frac{\tau}{\e}\right)v\right)_{L_2(\Om_\e)}+ \left(\frac{\p
u}{\p x_2},\frac{\p v}{\p x_2}\right)_{L_2(\Om_\e)}
\end{equation*}
on $\Hoper^1(\Om_\e,\Gp_+\cup\gp_\e)$, where $\tau\in[-1,1)$. Here
$\Hoper^1(\Om_\e,\Gp_+\cup\gp_\e)$ is the set of the functions in
$\Ho^1(\Om_\e,\Gp_+\cup\gp_\e)$ satisfying periodic boundary
conditions on the lateral boundaries of $\Om_\e$. The operator
$\Hpe(\tau)$ has a compact resolvent, since it is bounded as that
from $L_2(\Om_\e)$ into $\H^1(\Om_\e)$, and the space $\H^1(\Om_\e)$
is compactly embedded into $L_2(\Om_\e)$. Hence, the spectrum of
$\Hpe(\tau)$ consists of its discrete part only. We denote the
eigenvalues of $\Hpe(\tau)$ by $\l_n(\tau,\e)$ and arrange them in
the ascending order with the multiplicities taking into account
\begin{equation*}
\l_1(\tau,\e)\leqslant\l_2(\tau,\e) \leqslant \ldots \leqslant
\l_n(\tau,\e)\leqslant \ldots
\end{equation*}
By \cite[Lm. 4.1]{BC} we know that
\begin{equation*}
\spec(\mathcal{H}_\e)=\essspec(\mathcal{H}_\e)=\bigcup\limits_{n=1}^\infty
\{\l_n(\tau,\e): \tau\in[-1,1)\},
\end{equation*}
where $\spec(\cdot)$ and $\essspec(\cdot)$ indicate the spectrum and
the essential spectrum of an operator.

By $\mathfrak{L}_\e$ we denote the subspace of $L_2(\Om_\e)$
consisting of the functions independent of $x_1$, and we shall make
use the decomposition
\begin{equation*}
L_2(\Om_\e)=\mathfrak{L}_\e\oplus \mathfrak{L}_\e^\bot,
\end{equation*}
where $\mathfrak{L}_\e^\bot$ is the orthogonal complement to
$\mathfrak{L}_\e$ in $L_2(\Om_\e)$. Let $\mathcal{Q}_\mu$ be the
self-adjoint non-negative operator in $\mathfrak{L}_\e$ associated
with the sesquilinear form
\begin{equation*}
\mathfrak{q}[u,v]:=\left(\frac{d u}{dx_2},\frac{d
v}{dx_2}\right)_{L_2(0,\pi)}+\mu u(0)\overline{v(0)} \quad
\text{on}\quad \Ho^1((0,\pi),\{\pi\}),
\end{equation*}
i.e., $\mathcal{Q}_\mu$ is the operator $-\frac{d^2}{dx_2^2}$ in $L_2(0,\pi)$ with the domain consisting of the functions in
$\H^2(0,\pi)$ satisfying the boundary conditions
\begin{equation*}
u(\pi)=0,\quad u'(0)-\mu u(0)=0.
\end{equation*}

Our next results are on the uniform resolvent convergence for
$\Hpe(\tau)$ and two-terms asymptotics for the first band functions.
\begin{thm}\label{th1.3}
Let $|\tau|<1-\k$, where $0<\k<1$ is a fixed constant and suppose
(\ref{1.5}). Then for sufficiently small $\e$ the estimate
\begin{equation}\label{1.14}
\left\| \left(\Hpe(\tau)-\frac{\tau^2}{\e^2}\right)^{-1} -
\mathcal{Q}_\mu^{-1}\oplus 0\right\|_{L_2(\Om_\e)\to
L_2(\Om_\e)}\leqslant C \k^{-1/2}(\e^{1/2}\mu+\e)
\end{equation}
holds true, where the constant $C$ is independent of $\e$, $\mu$,
and $\k$.
\end{thm}

\begin{thm}\label{th1.4}
Let the hypothesis of Theorem~\ref{th1.3} holds true. Then given any
$N$, for $\e<2\k^{1/2}N^{-1}$ the eigenvalues $\l_n(\tau,\e)$,
$n=1,\ldots,N$, satisfy the relations
\begin{equation}\label{1.15}
\begin{aligned}
&\l_n(\tau,\e)=\frac{\tau^2}{\e^2}+\L_n(\mu)+R_n(\tau,\e,\mu),
\\
&|R_n(\tau,\e,\mu)|\leqslant C \k^{-1/2} n^4\e^{1/2}\mu,
\end{aligned}
\end{equation}
where $\L_n(\mu)$, $n=1,\ldots,N$, are first $N$ eigenvalues of
$\mathcal{Q}_\mu$, and the constant $C$ is the same as in
(\ref{1.14}). The eigenvalues $\L_n(\mu)$ solve the equation
\begin{equation}\label{1.17}
\sqrt{\L} \cos\sqrt{\L}\pi+\mu\sin\sqrt{\L}\pi=0,
\end{equation}
are holomorphic w.r.t. $\mu$, and
\begin{equation}\label{1.16}
\L_n(\mu)=\left(n-\frac{1}{2}\right)^2+\frac{\mu}{\pi\left(n-\frac{1}{2}\right)}
+\Odr(\mu^2).
\end{equation}
\end{thm}

Let
\begin{equation}\label{1.22}
\tht(\b):= -\sum\limits_{j=1}^{+\infty}
\frac{1}{n\sqrt{4j^2-\b}(2j+\sqrt{4j^2-\b})}.
\end{equation}
It will be shown in Lemma~\ref{lm4.4} that the function $\tht(\b)$
is holomorphic in $\b$ and its Taylor series is
\begin{equation}\label{4.58a}
\tht(\b)= -\sum\limits_{j=1}^{+\infty}
\frac{(2j-1)!!\z(2j+1)}{8^j\,j!}\b^{j-1},
\end{equation}
where $\z$ is the Riemann zeta-function.

 Our last main result provides the asymptotic expansion for the
bottom of the essential spectrum of $\mathcal{H}_\e$.

\begin{thm}\label{th1.5}
For $\e$ small enough, the first eigenvalue $\l_1(\tau,\e)$ attains
its minimum at $\tau=0$,
\begin{equation}\label{1.18}
\inf\limits_{\tau\in[-1,1)}\l_1(\tau,\e)=\l_1(0,\e).
\end{equation}
The asymptotics
\begin{equation}\label{1.19}
\l_1(0,\e)=\L(\e,\mu)+\Odr(\mu\e^{-1/2}\E^{-2\e^{-1}}+\e^{1/2}\eta^{1/2})
\end{equation}
holds true, where $\L(\e,\mu)$ is the real solution to the equation
\begin{equation}\label{4.69}
\sqrt{\L}\cos\sqrt{\L}\pi+\mu\sin\sqrt{\L\pi}
-\e^3\mu\L^{3/2}\tht(\e^2 \L)\cos\sqrt{\L}\pi=0
\end{equation}
satisfying the restriction
\begin{equation}\label{4.50}
\L(\e,\mu)=\L_1(\mu)+o(1),\quad\e\to0.
\end{equation}
The function $\L(\e,\mu)$ is jointly holomorphic w.r.t. $\e$ and
$\mu$ and can be represented as the series
\begin{equation}\label{1.23}
\L(\e,\mu)=\L_1(\mu)+\mu^2\sum\limits_{j=1}^{+\infty}\e^{2j+1}
K_{2j+1}(\mu)+\mu^3\sum\limits_{j=2}^{+\infty}\e^{2j} K_{2j}(\mu),
\end{equation}
where the functions $K_j(\mu)$ are holomoprhic w.r.t. $\mu$, and, in
particular,
\begin{equation}\label{1.20}
\begin{aligned}
&K_3(\mu)=-\frac{\z(3)}{4} \frac {\L_1^2(\mu)}{\pi\L_1(\mu)
+\mu+\pi\mu^2},
\\
&K_4(\mu)=0,
\\
&K_5(\mu)=-\frac{3\z(5)}{64} \frac{\L_1^3(\mu)} { \pi\L_1(\mu)+\mu+
\pi\mu^2 },
\\
&K_6(\mu)=\frac{\z(3)^2}{64}\frac{\L_1^3(\mu)
(2\pi^2\L_1^2(\mu)+7\pi\mu\L_1(\mu)+2\pi^2\mu^2\L_1(\mu)+ 7\mu^2+
7\pi\mu^{3})}{(\pi\L_1(\mu)+\mu+\pi\mu^2)^3}
\\
&K_7(\mu)=-\frac {5\z(7)}{512}\frac { \L_1^4(\mu)}{\pi
\L_1(\mu)+\mu+\pi\mu^2},
\\
&K_8(\mu)=\frac {3\z(3)\z(5)}{512} \frac {\L_1^4(\mu) (2\pi
^2\L_1^2(\mu)+9\pi\mu
\L_1(\mu)+2\pi^2\mu^2\L_1(\mu)+9\mu^2+9\mu^3\pi
)}{(\pi\L_1(\mu)+\mu+\pi\mu^2)^3}.
\end{aligned}
\end{equation}
The asymptotic expansion for the associated eigenfunction of
$\Hpe(0)$ reads as follows,
\begin{equation}\label{1.25}
\|\po(\cdot,\e)-\Po_\e\|_{\H^1(\Om_\e)}=\Odr(\mu\E^{-2\e^{-1}}+\e\eta^{1/2}),
\end{equation}
where the function $\Po_\e$ is defined in (\ref{5.30a}).
\end{thm}

\begin{rem}
All other coefficients of (\ref{1.23}) can be determined recursively
by substituting this series and (\ref{4.58a}) into (\ref{4.69}),
expanding then (\ref{4.69}) in powers of $\e$, and solving the
obtained equations w.r.t. $K_i$.
\end{rem}

\section{Uniform resolvent convergence for $\mathcal{H}_\e$}

In this section we prove Theorem~\ref{th1.1}. Given a function $f\in
L_2(\Om)$, we denote
\begin{equation*}
u_\e:=(\mathcal{H}_\e-\iu)^{-1}f,\quad
u^{(\mu)}:=(\mathcal{H}^{(\mu)}-\iu)^{-1}f.
\end{equation*}
The main idea of the proof is to construct a special corrector
$W=W(x,\e,\mu)$ with certain properties and to estimate the norms of
$v_\e:=u_\e-(1+W) u^{(\mu)}$ and $u^{(\mu)}W$. In fact, the function
W reflects the geometry of the alternation of the boundary
conditions for $\mathcal{H}_\e$, and this is why it is much simpler
to estimate independently $v_\e$ and $u^{(\mu)}W$ than trying to get
directly the estimate for $u_\e-u^{(\mu)}$ and $u_\e-u^{(0)}$. Next
lemma is the first main ingredient in the proof of
Theorem~\ref{th1.1} and it shows how $W$ is employed.

\begin{lem}\label{lm2.1}
Let $W=W(x,\e,\mu)$ be an $\e\pi$-periodic in $x_1$ function
belonging to $C(\overline{\Om})\cap
C^\infty(\overline{\Om}\setminus\{x: x_2=0, \, x_1=\pm\e\eta+\e\pi
n,\, n\in \mathds{Z}\})$ satisfying boundary conditions
\begin{equation}\label{2.1}
W=-1\quad\text{on}\quad \g_\e,\qquad \frac{\p W}{\p
x_2}=-\mu\quad\text{on}\quad \G_\e,
\end{equation}
and having differentiable asymptotics
\begin{equation}\label{2.2}
W(x,\e,\mu)=c_\pm(\e,\mu)
r_\pm^{1/2}\sin\frac{\tht_\pm}{2}+\Odr(\r_\pm),\quad r_\pm\to+0.
\end{equation}
Here $(r_\pm,\tht_\pm)$ are polar coordinates centered at
$(\pm\e\eta,0)$ such that the values $\tht_\pm=0$ correspond to the
points of $\g_\e$. Assume also that $\D W\in C(\overline{\Om})$.
Then $(1+W)u^{(\mu)}$ belongs to $\Ho^1(\Om,\G_+\cup\g_\e)$, and
\begin{equation}\label{2.3}
\begin{aligned}
\|\nabla v_\e&\|_{L_2(\Om)}^2+\iu\|v_\e\|_{L_2(\Om)}^2= (f,v_\e
W)_{L_2(\Om)}+(u^{(\mu)}\D W,v_\e)_{L_2(\Om)}
\\
&-2\iu(u^{(\mu)}W, v_\e)_{L_2(\Om)} -2(W\nabla u^{(\mu)}, \nabla
v_\e)_{L_2(\Om)}-\mu(u^{(\mu)}, Wv_\e)_{L_2(\G_\e)}.
\end{aligned}
\end{equation}
\end{lem}

\begin{proof}
We write the integral identities for $u_\e$ and $u^{(\mu)}$,
\begin{equation}\label{2.4}
(\nabla u_\e,\nabla \phi)_{L_2(\Om)}+\iu (u_\e,\phi)_{L_2(\Om)}=
(f,\phi)_{L_2(\Om)}
\end{equation}
for all $\phi\in \Ho^1(\Om,\G_+\cup\g_\e)$, and
\begin{equation}
(\nabla u^{(\mu)},\nabla
\phi)_{L_2(\Om)}+\mu(u^{(\mu)},\phi)_{L_2(\G_-)}+ \iu
(u^{(\mu)},\phi)_{L_2(\Om)}= (f,\phi)_{L_2(\Om)} \label{2.5}
\end{equation}
for all $\phi\in \Ho^1(\Om,\G_+)$. Employing the smoothness of $W$,
(\ref{2.1}), (\ref{2.2}), and proceeding as in the proof of
Lemma~3.2 in \cite{BC}, we check that $(1+W)\phi\in
\Ho^1(\Om,\G_+\cup\g_\e)$, if $\phi$ belongs to the domain of
$\mathcal{H}_\e$ or $\mathcal{H}^{(\mu)}$. Hence, $(1+W)u^{(\mu)}\in
\Ho^1(\Om,\G_+\cup\g_\e)$. Thus,
\begin{equation}\label{2.6}
(1+W)v_\e\in \Ho^1(\Om,\G_+\cup\g_\e).
\end{equation}
We take $\phi=(1+W)v_\e$ in (\ref{2.5}),
\begin{align*}
(\nabla u^{(\mu)},&\nabla (1+W)v_\e)_{L_2(\Om)}+
\mu(u^{(\mu)},(1+W)v_\e)_{L_2(\G_-)}
\\
& +\iu (u^{(\mu)},(1+W)v_\e)_{L_2(\Om)}= (f,(1+W)v_\e)_{L_2(\Om)},
\\
(\nabla u^{(\mu)},&(1+W)\nabla v_\e)_{L_2(\Om)}+ \iu
(u^{(\mu)},(1+W)v_\e)_{L_2(\Om)}=
\\
&(f,(1+W)v_\e)_{L_2(\Om)}- (\nabla u^{(\mu)},v_\e\nabla
W)_{L_2(\Om)}-\mu(u^{(\mu)},(1+W)v_\e)_{L_2(\G_-)},
\\
(\nabla (1 + &W) u^{(\mu)},\nabla v_\e)_{L_2(\Om)}+ \iu
((1+W)u^{(\mu)},v_\e)_{L_2(\Om)}=
\\
&(f,(1+W)v_\e)_{L_2(\Om)}- (\nabla u^{(\mu)},v_\e\nabla
W)_{L_2(\Om)}
\\
&+ (u^{(\mu)}\nabla W,\nabla v_\e)_{L_2(\Om)}
-\mu(u^{(\mu)},(1+W)v_\e)_{L_2(\G_-)}.
\end{align*}
We deduct (\ref{2.4}) with $\phi=v_\e$ from the last identity,
\begin{equation}\label{2.7}
\begin{aligned}
\|\nabla
v_\e\|_{L_2(\Om)}^2+\iu&\|v_\e\|_{L_2(\Om)}^2=-(f,Wv_\e)_{L_2(\Om)}
+(\nabla u^{(\mu)},v_\e\nabla W)_{L_2(\Om)}
\\
& - (u^{(\mu)}\nabla W,\nabla v_\e)_{L_2(\Om)}
+\mu(u^{(\mu)},(1+W)v_\e)_{L_2(\G_-)}.
\end{aligned}
\end{equation}
We integrate by parts taking into account (\ref{2.1}), (\ref{2.5}),
and (\ref{2.6}),
\begin{align*}
(\nabla u^{(\mu)},& v_\e\nabla W)_{L_2(\Om)}-(u^{(\mu)}\nabla
W,\nabla v_\e)_{L_2(\Om)}
\\
&=(\nabla u^{(\mu)},v_\e\nabla W)_{L_2(\Om)}+\int\limits_{\G_\e}
u^{(\mu)}\frac{\p W}{\p x_2}\overline{v}_\e\di x_1+ (\mathrm{div}\,
u^{(\mu)}\nabla W,v_\e)_{L_2(\Om)}
\\
&=2(\nabla u^{(\mu)},v_\e\nabla W)_{L_2(\Om)}-\mu
(u^{(\mu)},v_\e)_{L_2(\G_\e)}+(u^{(\mu)}\D W,v_\e)_{L_2(\Om)},
\end{align*}
and
\begin{align*}
(\nabla u^{(\mu)},& v_\e\nabla W)_{L_2(\Om)}=(\nabla
u^{(\mu)},\nabla W v_\e)_{L_2(\Om)}- (\nabla u^{(\mu)},W \nabla
v_\e)_{L_2(\Om)}
\\
&\hphantom{v_\e\nabla W)_{L_2(\Om)}}=(f,Wv_\e)_{L_2(\Om)}
-\iu(u^{(\mu)},W v_\e)_{L_2(\Om)}
\\
&\hphantom{v_\e\nabla W)_{L_2(\Om)}=}-\mu (u^{(\mu)},W
v_\e)_{L_2(\Gp_-)}-(\nabla u^{(\mu)}, W\nabla v_\e)_{L_2(\Om)}.
\end{align*}
We substitute the obtained identities into (\ref{2.7}) and this
completes the proof.
\end{proof}

As it follows from (\ref{2.3}), to prove the smallness of $v_\e$ in
$\H^1(\Om)$-norm, it is sufficient to construct a function $W$
satisfying the hypothesis of Lemma~\ref{lm2.1} so that the
quantities $W$ and $\D W$ are small in certain sense. This is why we
 introduce $W$ as a formal asymptotic
solution to the equation
\begin{equation}\label{2.8}
\D W=0\quad \text{in}\quad \Om,
\end{equation}
satisfying (\ref{2.1}), (\ref{2.2}) and other assumptions of
Lemma~\ref{lm2.1}. To construct such solution, we shall employ the
asymptotic constructions from \cite{AsAn}, \cite{GDu99} based on the
method of matching of asymptotic expansions \cite{Il} and the
boundary layer method \cite{VL}. We also mention that similar
approach was used in \cite[Lm. 1]{Gzh01} for constructing a
different corrector.

First we construct $W$ formally, and after that we shall prove
rigourously all the required properties of the constructed
corrector. Denote $\xi=(\xi_1,\xi_2)=x\e^{-1}$,
$\vs^{(j)}=(\vs^{(j)}_1,\vs^{(j)}_2)$, $\vs_1^{(j)}=(\xi_1-\pi
j)\eta^{-1}$, $\vs^{(j)}_2=\xi_2\eta^{-1}$. Outside a small
neighborhood of $\g_\e$ we construct $W$ as a boundary layer
\begin{equation*}
W(x,\e,\mu)=\e\mu X(\xi).
\end{equation*}
We pass to $\xi$ in (\ref{2.8}) and let $\eta=0$ in the boundary
conditions. It yields a boundary value problem for $X$,
\begin{equation}\label{2.11}
\D_\xi X=0,\quad \xi_2>0,\qquad \frac{\p X}{\p\xi_2}=-1, \quad
\xi\in\G^0:=\{\xi: \xi_2=0\} 
\setminus\bigcup\limits_{j=-\infty}^{+\infty}
\{(\pi j,0)\},
\end{equation}
where the function $X$ should be $\pi$-periodic in $\xi_1$ and decay
exponentially as $\xi_2\to+\infty$. It was shown in \cite{GD98} that
the required solution to (\ref{2.11}) is
\begin{equation*}
X(\xi):=\RE\ln\sin (\xi_1+\iu \xi_2)+\ln 2-\xi_2.
\end{equation*}
It was also shown that
\begin{equation*}
X\in C^\infty(\{\xi: \xi_2\geqslant 0,\, \xi\not=(\pi j,0),\, j\in
\mathds{Z}\}),
\end{equation*}
and this function satisfies the differentiable asymptotics
\begin{equation}\label{2.12}
X(\xi)=\ln|\xi-(\pi j,0)|+\ln 2-\xi_2+\Odr(
\xi-(\pi
j,0)|^2),\quad \xi\to(\pi j,0),\quad j\in\mathds{Z}.
\end{equation}
In view of the last identity we rewrite the asymptotics for $X$ as
$\xi\to(\pi j,0)$ in terms of $\vs^{(j)}$,
\begin{equation}\label{2.13}
\begin{aligned}
\e\mu X(\xi)=&\e\mu \big(\ln|\xi-(\pi j,0)|+\ln 2-\xi_2\big)+
\Odr(
\e\mu|\xi-(\pi j,0)|^2)
\\
=&-1+\e\mu \big(\ln|\vs^{(j)}|+\ln
2\big)-\e\mu\eta\vs^{(j)}_2+\Odr(\e\mu\eta^2|\vs^{(j)}|^2).
\end{aligned}
\end{equation}
In accordance with the method of matching of asymptotic expansions
it follows from the obtained identities that in a small neighborhood
of each interval of $\g_\e$ we should construct $W$ as an internal
layer,
\begin{equation}\label{2.14}
W(x,\e,\mu)=-1+\e\mu W_{in}^{(j)}(\vs^{(j)}),
\end{equation}
where
\begin{equation}\label{2.15}
W_{in}^{(j)}(\vs^{(j)})=\ln|\vs^{(j)}|+\ln 2+o(1),\quad
\vs^{(j)}\to+\infty.
\end{equation}
We substitute (\ref{2.14}) into (\ref{2.8}), (\ref{2.1}), which
leads us to the boundary value problem for $W_{in}^{(j)}$,
\begin{equation}\label{2.16}
\begin{aligned}
&\D_{\vs^{(j)}} W_{in}^{(j)}=0,\quad \vs_2^{(j)}>0,
\\
&W_{in}^{(j)}=0,\quad \vs^{(j)}\in\g^1,\qquad \frac{\p
W_{in}^{(j)}}{\p\vs_2^{(j)}}=0,\quad \vs^{(j)}\in\G^1,
\\
&\g^1:=\{\vs:\, |\vs_1|<1,\,\vs_2=0\},\quad \G^1:=O\vs_1
\setminus\overline{\g}^1.
\end{aligned}
\end{equation}
It was shown in \cite{GD98} that the problem (\ref{2.15}),
(\ref{2.16}) is solvable and
\begin{equation}\label{2.17}
W_{in}^{(j)}(\vs^{(j)})=Y(\vs^{(j)}),\quad
Y(\vs):=\RE\ln(z+\sqrt{z^2-1}),\quad z=\vs_1+\iu \vs_2,
\end{equation}
where the branch of the root is fixed by the requirement
$\sqrt{1}=1$. It was also shown that
\begin{equation}\label{2.18}
Y(\vs)=\ln|\vs|+\ln 2+\Odr(|\vs|^{-2}),\quad \vs\to\infty.
\end{equation}
As it follows from the last asymptotics, the term
$-\e\mu\vs_2^{(j)}$ in (\ref{2.13}) is not matched with any term in
the boundary layer. At the same time, it was found in \cite{AsAn},
\cite{Gzh01}, \cite{GDu99} that such terms should be either matched
or cancelled out to obtain a reasonable estimate for the error
terms. This is also the case in our problem. In contrast to
\cite{AsAn}, \cite{Gzh01}, \cite{GDu99}, to solve this issue we
shall not construct additional terms in $W$, but employ a different
trick to solve this issue. Namely, we add the function $\e\mu\xi_2$
to the boundary layer and add also $-\mu x_2$ as the external
expansion. It changes neither equations nor boundary conditions for
$W$ but allows us to cancel out the mentioned term in (\ref{2.13}).
The final form of $W$ is as follows,
\begin{equation}\label{2.19}
\begin{aligned}
W(x,\e,\mu)=&-\mu x_2+\e\mu (X(\xi)+\xi_2)
\prod\limits_{j=-\infty}^{+\infty}
\Big(1-\chi_1\big(|\vs^{(j)}|\eta^\a\big)\Big)
\\
&+\sum\limits_{j=-\infty}^{+\infty} \chi_1\big(
|\vs^{(j)}|\eta^\a\big)\big(-1+\e\mu Y(\vs^{(j)})\big),
\end{aligned}
\end{equation}
where $\a\in(0,1)$ is a constant, which will be chosen later, and
$\chi_1=\chi_1(t)$ is an infinitely differentiable cut-off function
taking values in $[0,1]$, being one as $t<1$, and vanishing as
$t>3/2$. It can be easily seen that the sum and the product in the
definition of (\ref{2.19}) are always finite.

Let us check that the function $W$ satisfies the hypothesis of
Lemma~\ref{lm2.1}. By direct calculations we check that the function
$W$ is $\e\pi$-periodic w.r.t. $x_1$, belongs to
$C(\overline{\Om})\cap C^\infty(\overline{\Om}\setminus\{x: x_2=0,
\, x_1=\pm\e\eta+\e\pi n,\, n\in \mathds{Z}\})$, and satisfies
(\ref{2.2}). The boundary condition on $\g_\e$ in (\ref{2.1}) is
obviously satisfied. Taking into account the boundary conditions
(\ref{2.11}), (\ref{2.15}), we check
\begin{align*}
\frac{\p W}{\p x_2}\Big|_{x\in\G_\e}=&-\mu+\e\mu \left( \frac{\p
X}{\p\xi_2}\Big|_{\xi\in\G^0}+1\right)\prod\limits_{j=-\infty}^{+\infty}
\big(1-\chi_1(|\vs^{(j)}|\eta^\a)\big)
\\
&+\e\mu \sum\limits_{j=-\infty}^{+\infty} \chi_1\big(
|\vs^{(j)}|\eta^\a\big)\frac{\p
Y}{\p\vs_2^{(j)}}\Big|_{\vs^{(j)}\in\G^1}=-\mu,
\end{align*}
i.e., the boundary condition on $\G_\e$ in (\ref{2.1}) is satisfied,
too.

Let us calculate $\D W$. In order to do it, we employ the equations
in (\ref{2.11}), (\ref{2.15}),
\begin{equation}\label{2.20}
\begin{aligned}
&\D W(x)=2\sum\limits_{j=-\infty}^{+\infty} \nabla_x \chi_1\big(
|\vs^{(j)}|\eta^\a\big)\cdot \nabla_x W_{mat}^{(j)}(x,\e,\mu)
\\
&\hphantom{\D W(x)=}+\sum\limits_{j=-\infty}^{+\infty}
W_{mat}^{(j)}(x,\e,\mu)\D_x \chi_1\big( |\vs^{(j)}|\eta^\a\big),
\\
&W_{mat}^{(j)}(x,\e,\mu)=-1+\e\mu
\big(Y(\vs^{(j)})-X(\xi)-\xi_2\big).
\end{aligned}
\end{equation}
It follows from the definition of $\xi$, $\vs^{(j)}$, $\chi_1$, $X$,
$Y$, and the last formula that $\D W\in C^\infty(\overline{\Om})$.
Thus, we can apply Lemma~\ref{lm2.1}. To estimate the right hand
side of (\ref{2.3}) we need two auxiliary lemmas.

Given any $\d\in(0,\pi/2)$, denote
\begin{equation*}
\Om^\d:=\bigcup\limits_{j=-\infty}^{+\infty} \Om_j^\d, \quad
\Om_j^\d:=\{x: |x-(\pi j,0)|<\e\d\}\cap\Om.
\end{equation*}

\begin{lem}\label{lm2.2}
For any $u\in\H^1(\Om)$ and any $\d\in(0,\pi/4)$ the inequality
\begin{equation}\label{2.21}
\|u\|_{L_2(\Om^\d)}\leqslant C\d \big(|\ln\d|^{1/2}+1\big)
\|u\|_{\H^1(\Om)}
\end{equation}
holds true, where the constant $C$ is independent of $\d$ and $u$.
\end{lem}

\begin{proof}
We begin with the formulas
\begin{align}
&\|u\|_{L_2(\Om^\d)}^2=\sum\limits_{j=-\infty}^{+\infty}
\|u\|_{L_2(\Om_j^\d)}^2, \label{2.22}
\\
&\|u\|_{L_2(\Om_j^\d)}^2=\int\limits_{\Om_j^\d} |u(x)|^2\di x=\e^2
\int\limits_{|\xi-(\pi j,0)|<\d,\,\xi_2>0} |u(\e\xi)|^2\di
\xi\nonumber
\\
&\hphantom{\|u\|_{L_2(\Om_j^\d)}^2}=\e^2 \int\limits_{|\xi-(\pi
j,0)|<\d,\,\xi_2>0}|\chi_2(\xi-(\pi j,0))u(\e\xi)|^2\di\xi,
\nonumber
\end{align}
where $\chi_2=\chi_2(\xi)$ is an infinitely differentiable function
being one as $|\xi|<\d$ and vanishing as $|\xi|>\pi/3$. We also
suppose that the functions $\chi_2$, $\chi_2'$ are bounded uniformly
in $\xi$ and $\d$. Hence,
\begin{equation*}
\chi_2(\cdot-(\pi j,0))u\in \Ho^1(\Pi_j^1,\p\Pi_j^1),\quad
\Pi_j^1:=\left\{\xi:\,|\xi_1-\pi
j|<\frac{\pi}{2},\,0<\xi_2<1\right\}.
\end{equation*}
By \cite[Lm. 3.2]{OSaHY}, we obtain
\begin{align*}
\e^2&\int\limits_{|\xi-(\pi j,0)|<\d,\,\xi_2>0} |\chi_2
u|^2\di\xi\leqslant C\e^2 \d^2(|\ln\d|+1) \int\limits_{\Pi_j^1}
\big(|\nabla_\xi \chi_2 u|^2+|\chi_2 u|^2\big)\di \xi
\\
&\leqslant C\e^2\d^2(|\ln\d|+1) \big(\|\nabla_\xi
u\|_{L_2(\Pi_j^1)}+ \|u\|_{L_2(\Pi_j^1)}\big)
\\
&\leqslant C\d^2(|\ln\d|+1) \|u\|_{\H^1(\{x: |x_1-\e\pi j|<\e\pi/2,
0<x_2<\pi\})}^2,
\end{align*}
where the constants $C$ are independent of $j$, $\e$, $\d$, $\mu$,
and $u$. We substitute these inequalities into (\ref{2.22}) and
arrive at (\ref{2.21}).
\end{proof}

\begin{lem}\label{lm2.3}
For any $u\in \H^2(\Om)$ and any $\d\in(0,\pi/2)$ the inequality
\begin{equation*}
\|u\|_{L_2(\g_\e^\d)}\leqslant C\d^{1/2}\|u\|_{\H^2(\Om)}, \quad
\g_\e^\d:=\{x: |x_1-\e\pi j|<\e\d, \, x_2=0\},
\end{equation*}
holds true, where the constant $C$ is independent of $\e$, $\d$, and
$u$.
\end{lem}

\begin{proof}
It is clear that
\begin{equation}
\|u\|_{L_2(\g_\e^\d)}=\sum\limits_{j=-\infty}^{+\infty}
\|u\|_{L_2(\g_{\e,j}^\d)},\quad \g_{\e,j}^\d:=\left\{x: |x_1-\e\pi
j|<\e\d,\, x_2=0\right\}. \label{2.24}
\end{equation}
It follows from the definition of $\chi_2$ (see the proof of
Lemma~\ref{lm2.2}) that
\begin{equation}\label{2.25}
\|u\|_{L_2(\g_{\e,j}^\d)}^2=\int\limits_{\g_{\e,j}^\d}
\left|\chi_2\left(\frac{x_1}{\e}-\pi j\right)u(x_1,0)\right|^2 \di
x_1.
\end{equation}
Since
\begin{equation*}
\chi_2\left(\frac{x_1}{\e}-\pi j\right)u(x_1,0)=\int\limits_{\e\pi
j-\frac{\e\pi}{2}}^{x_1} \frac{\p}{\p x_1} \left(
\chi_2\left(\frac{x_1}{\e}-\pi j\right)u(x_1,0) \right)\di x_1,
\end{equation*}
by the Cauchy-Schwartz inequality we get
\begin{gather*}
 \frac{\p}{\p x_1} \left( \chi_2\left(\frac{x_1}{\e}-\pi
j\right)u(x_1,0) \right)= \chi_2\left(\frac{x_1}{\e}-\pi
j\right)\frac{\p u}{\p x_1}(x_1,0)+\e^{-1}
\chi_2'\left(\frac{x_1}{\e}-\pi j\right)u(x_1),
\\
\left|\chi_2\left(\frac{x_1}{\e}-\pi
j,0\right)u(x_1,0)\right|^2\leqslant C \left(\e
\int\limits_{\g_{\e,j}} \left|\frac{\p u}{\p x_1}(x_1,0)\right|^2\di
x_1+\e^{-1} \int\limits_{\g_{\e,j}} |u(x_1,0)|^2\di x_1\right),
\\
\g_{\e,j}:=\left\{x: |x_1-\e\pi j|<\frac{\e\pi}{2},\, x_2=0\right\},
\end{gather*}
where the constants $C$ are independent of $j$, $\e$, $\d$, and $u$.
The last estimate and (\ref{2.25}) imply
\begin{equation*}
\|u\|_{L_2(\g_{\e,j}^\d)}^2\leqslant C\d \left( \Big\|\frac{\p u}{\p
x_1}\Big\|_{L_2(\g_{\e,j})}^2+\|u\|_{L_2(\g_{\e,j})}^2 \right),
\end{equation*}
where the constant $C$ is independent of $j$, $\e$, $\d$, and $u$.
We substitute the obtained inequality into (\ref{2.24}) and employ
the standard embedding of $\H^2(\Om)$ into $\H^1(\G_-)$ that
completes the proof.
\end{proof}

\begin{lem}\label{lm2.4}
The estimates
\begin{align}
&|\D W|\leqslant C\e^{-1}\mu(1+\eta^{4\a-2}), && x\in\Om,
\label{2.26}
\\
&|W|\leqslant C\e\mu(|\ln\d|+1), && x\in\Om\setminus\Om^\d, &&
\frac{3}{2}\eta^\a<\d<\frac{\pi}{2}, \label{2.27}
\\
&|W|\leqslant C, && x\in\Om^\d, &&
\frac{3}{2}\eta^\a<\d<\frac{\pi}{2}, \label{2.28}
\end{align}
are valid, where the constants $C$ are independent of $\e$, $\mu$,
$\eta$, $\d$, and $x$.
\end{lem}

\begin{proof}
Since $W$ is $\e\pi$-periodic w.r.t. $x_1$, it is sufficient to
prove the estimates only for $|x_1|<\e\pi/2$, $0<x_2<\pi$. It
follows directly from the definition of $X$, $Y$, and (\ref{2.15}),
(\ref{2.18}) that for any $\d\in(0,\pi/2)$
\begin{align*}
&|X(\xi)|\leqslant C\big(|\ln\d|+1\big), \quad
|\xi_1|<\frac{\pi}{2}, \quad \xi_2>0, \quad |\xi|\geqslant \d,
\\
&|Y(\vs)|\leqslant C \big(|\ln \d\eta^{-1}|+1\big)\leqslant C
\big(|\ln\d|+\e^{-1}\mu^{-1}\big),\quad |\vs|\leqslant \d\eta^{-1},
\end{align*}
where the constants $C$ are independent of $\e$, $\mu$, $\eta$,
$\d$, and $x$. These estimates and (\ref{2.19}) imply (\ref{2.27}),
(\ref{2.28}).

It follows from the definition of $\chi_1$ that $\D W$ is non-zero
only as
\begin{equation*}
\eta^{-\a}<|\vs^{(1)}|<\frac{3}{2}\eta^{-\a}.
\end{equation*}
For the corresponding values of $x$ due to (\ref{2.15}),
(\ref{2.17}) the differentiable asymptotics
\begin{equation*}
W_{mat}^{(1)}(x,\e,\mu)=\Odr\big(\e\mu(|\vs^{(1)}|^{-2}+|\xi|^2)\big),
\quad \eta^{-\a}<|\vs^{(1)}|<\frac{3}{2}\eta^{-\a},\quad
\eta^{1-\a}<|\xi|<\frac{3}{2}\eta^{1-\a},
\end{equation*}
holds true. Hence, for the same values of $\xi$ and $\vs^{(1)}$
\begin{align*}
&W_{mat}^{(1)}=\Odr\big(\e\mu (\eta^{2\a}+\eta^{2-2\a}) \big),
\\
&\nabla_x
W_{mat}^{(1)}=\Odr\big(\mu(\eta^{-1}|\vs^{(1)}|^{-3}+|\xi|)
\big)=\Odr\big(\mu(\eta^{1-\a}+\eta^{3\a-1})\big).
\end{align*}
Substituting the identities obtained into (\ref{2.20}) and taking
into account the relations
\begin{equation*}
\nabla_x \chi_1\big(
|\vs^{(j)}|\eta^\a\big)=\Odr(\e^{-1}\eta^{\a-1}),\quad \D_x
\chi_1\big( |\vs^{(j)}|\eta^\a\big)=\Odr(\e^{-2}\eta^{2\a-2}),
\end{equation*}
we arrive at (\ref{2.26}).
\end{proof}

Let us estimate the right hand side of (\ref{2.3}). We have
\begin{align}
&|(f,Wv_\e)_{L_2(\Om)}|\leqslant \|f\|_{L_2(\Om)}\|W
v_\e\|_{L_2(\Om)},\nonumber
\\
&\|W v_\e\|_{L_2(\Om)}^2=\|W v_\e\|_{L_2(\Om\setminus\Om^\d)}^2+ \|W
v_\e\|_{L_2(\Om^\d)}^2.\label{2.28a}
\end{align}
Let $\d\in\left(\frac{3}{2}\eta^\a,\frac{\pi}{2}\right)$. Applying
Lemma~\ref{lm2.2} and using (\ref{2.27}), (\ref{2.28}), we have
\begin{equation}
\begin{aligned}
&\|v_\e W\|_{L_2(\Om\setminus\Om^\d)}^2\leqslant C
\e^2\mu^2(|\ln\d|^2+1)\|v_\e\|_{L_2(\Om\setminus\Om^\d)}^2,
\\
&\|v_\e W\|_{L_2(\Om^\d)}^2\leqslant C\d^2
(|\ln\d|+1)\|v_\e\|_{\H^1(\Om)}^2.
\end{aligned}\label{2.28b}
\end{equation}
Here and till the end of this section we indicate by $C$ various
non-essential constants independent of $\e$, $\mu$, $\eta$, $\d$,
$x$, $v_\e$, $u^{(\mu)}$, and $f$. The inequalities (\ref{2.28b})
yield
\begin{equation}\label{2.29}
|(f,v_\e W)_{L_2(\Om)}|\leqslant
C\big(\e\mu|\ln\d|+\d|\ln\d|^{1/2}+\d\big)
\|v_\e\|_{\H^1(\Om)}\|f\|_{L_2(\Om)}.
\end{equation}

It follows from the definition of $u^{(\mu)}$ that
\begin{equation}\label{2.30}
\|u^{(\mu)}\|_{\H^2(\Om)}\leqslant C\|f\|_{L_2(\Om)}.
\end{equation}
Taking into account this inequality, we proceed in the same way as
in (\ref{2.28a}), (\ref{2.28b}), (\ref{2.29}),
\begin{align}
&
\begin{aligned}
\|u^{(\mu)}W\|_{L_2(\Om)} &\leqslant
C(\e\mu|\ln\d|+\d|\ln\d|^{1/2}+\d)\|u^{(\mu)}\|_{\H^1(\Om)}
\\
&\leqslant C(\e\mu|\ln\d|+\d|\ln\d|^{1/2}+\d)\|f\|_{L_2(\Om)},
\end{aligned}\label{2.36a}
\\
&
\begin{aligned}
\|W\nabla u^{(\mu)}\|_{L_2(\Om)} &\leqslant
C(\e\mu|\ln\d|+\d|\ln\d|^{1/2}+\d)\|u^{(\mu)}\|_{\H^2(\Om)}
\\
&\leqslant C(\e\mu|\ln\d|+\d|\ln\d|^{1/2}+\d)\|f\|_{L_2(\Om)},
\end{aligned}\label{2.37a}
\\
&
\begin{aligned}
\big|(u^{(\mu)},W v_\e)_{L_2(\Om)}+& (\nabla u^{(\mu)},W\nabla
v_\e)_{L_2(\Om)}\big|
\\
&\leqslant \|u^{(\mu)}W\|_{L_2(\Om)}\|v_\e\|_{L_2(\Om)} +\|W\nabla
u^{(\mu)}\|_{L_2(\Om)}\|\nabla v_\e\|_{L_2(\Om)}
\\
&\leqslant C(\e\mu|\ln\d|+\d|\ln\d|^{1/2}+\d)
\|f\|_{L_2(\Om)}\|v_\e\|_{\H^1(\Om)}.
\end{aligned}\label{2.34a}
\end{align}

Employing (\ref{2.26}) instead of (\ref{2.27}), (\ref{2.28}), and
applying then Lemma~\ref{lm2.2} with $\d=\eta^\a$, we get
\begin{equation}\label{2.33}
\begin{aligned}
\|u^{(\mu)}\D W\|_{L_2(\Om)}=\|u^{(\mu)}\D
W\|_{L_2(\Om_{2\eta^\a})}&\leqslant C
\eta^\a\e^{-3/2}\mu^{1/2}(1+\eta^{4\a-2}) \|u^{(\mu)}\|_{\H^1(\Om)}
\\
&\leqslant C \eta^\a\e^{-3/2}\mu^{1/2}(1+\eta^{4\a-2})
\|f\|_{L_2(\Om)}.
\end{aligned}
\end{equation}
Using (\ref{2.27}), (\ref{2.28}), (\ref{2.29}), Lemma~\ref{lm2.3}
with $\d=\widetilde{\d}\in(\eta^\a,\pi/2)$, the embedding of
$\H^2(\Om)$ in $\H^1(\G_-)$, and proceeding as in (\ref{2.28a}),
(\ref{2.28b}), (\ref{2.29}), we obtain
\begin{align}
&\big|(u^{(\mu)},W v_\e)_{L_2(\G_\e)}\big|\leqslant
\|u^{(\mu)}W\|_{L_2(\G_\e)}\|v_\e\|_{L_2(\G_-)}\leqslant C
\|u^{(\mu)}W\|_{L_2(\G_\e)}\|v_\e\|_{\H^1(\Om)},\nonumber
\\
&
\begin{aligned}
&\|u^{(\mu)}W\|_{L_2(\G_\e)}^2=
\|u^{(\mu)}W\|_{L_2(\G_\e\setminus\g_\e^{\widetilde{\d}})}^2
+\|u^{(\mu)}W\|_{L_2(\g_\e^{\widetilde{\d}})}^2
\\
&\hphantom{\|u^{(\mu)}W\|_{L_2(\G_\e)}^2} \leqslant
C\e^2\mu^2(|\ln\widetilde{\d}|^2+1) \|u^{(\mu)}\|_{L_2(\G_\e)}^2+
C\widetilde{\d} \|u^{(\mu)}\|_{\H^2(\Om)}^2
\\
&\hphantom{\|u^{(\mu)}W\|_{L_2(\G_\e)}^2}\leqslant
C\big(\widetilde{\d}+\e^2\mu^2(|\ln\widetilde{\d}|^2+1)\big)
\|f\|_{L_2(\Om)}^2,
\end{aligned}\label{2.35a}
\\
&\big|(u^{(\mu)},W v_\e)_{L_2(\G_\e)}\big|\leqslant C
\big(\widetilde{\d}^{1/2}+\e\mu(|\ln\widetilde{\d}|+1)\big)
\|f\|_{L_2(\Om)},\nonumber
\end{align}
Let $\a\in(1/2,1)$. The last obtained estimate, (\ref{2.29}),
(\ref{2.34a}), (\ref{2.33}), and (\ref{2.3}) yield
\begin{equation*}
\|v_\e\|_{\H^1(\Om)}^2\leqslant C (\d|\ln\d|^{1/2}+\e\mu|\ln\d|
+\e\mu^2|\ln\widetilde{\d}|+\mu\widetilde{\d}^{1/2})\|f\|_{L_2(\Om)}
\|v_\e\|_{\H^1(\Om)},
\end{equation*}
and it is assumed here that
\begin{equation*}
\eta^\a<\d<\pi/2,\quad \eta^\a<\widetilde{\d}<\pi/2, \quad
\d=\d(\e)\to+0,\quad \widetilde{\d}=\widetilde{\d}(\e)\to+0
\quad\text{as}\quad \e\to+0.
\end{equation*}
Thus, taking $\d=\e\mu$, $\widetilde{\d}=\e^2\mu^2$, we get
\begin{equation*}
\|v_\e\|_{\H^1(\Om)}\leqslant C\e\mu |\ln\e\mu|\|f\|_{L_2(\Om)},
\end{equation*}
and it proves (\ref{1.9}).

We take $\d=\e\mu$ in (\ref{2.36a}) and employ (\ref{1.9}),
\begin{align*}
\|(\mathcal{H}_\e-&\iu)^{-1}f-(\mathcal{H}^{(\mu)}-\iu)^{-1}f
\|_{L_2(\Om)}=\|u_\e-u^{(\mu)}\|_{L_2(\Om)}
\\
&\leqslant \|u_\e-(1+W)u^{(\mu)}\|_{L_2(\Om)}+\|u^{(\mu)}
W\|_{L_2(\Om)}
\\
&\leqslant C\e\mu|\ln\e\mu|\|f\|_{L_2(\Om)},
\end{align*}
which proves (\ref{1.10}).

\begin{lem}\label{lm2.5}
The estimate
\begin{equation}\label{2.35}
\|\nabla(u^{(\mu)}W)\|_{L_2(\Om)}\leqslant C
\mu^{1/2}\|f\|_{L_2(\Om)}
\end{equation}
holds true.
\end{lem}

\begin{proof}
We integrate by parts employing (\ref{2.1}), (\ref{2.2}),
(\ref{1.4}), (\ref{1.11}),
\begin{align*}
\|\nabla (u^{(\mu)}W)\|_{L_2(\Om)}^2=&-\left( \frac{\p}{\p x_2}
u^{(\mu)} W, u^{(\mu)}W \right)_{L_2(\G_-)}- \left(\D (u^{(\mu)}
W),u^{(\mu)}W\right)_{L_2(\Om)}
\\
=&-\mu\|u^{(\mu)}W\|_{L_2(\G_-)}^2 +\int\limits_{\g_\e}
|u^{(\mu)}|^2\frac{\p W}{\p x_2}\di x_1+ \mu
(u^{(\mu)},u^{(\mu)}W)_{L_2(\G_\e)}
\\
&-(W\D u^{(\mu)},W u^{(\mu)})_{L_2(\Om)}-2\left(W\nabla
u^{(\mu)},u^{(\mu)}\nabla W\right)_{L_2(\Om)}
\\
&-\left(u^{(\mu)}\D W, u^{(\mu)}W\right)_{L_2(\Om)}.
\end{align*}
We take the real part of this identity,
\begin{equation}\label{3.36}
\begin{aligned}
\|\nabla (u^{(\mu)}W)\|_{L_2(\Om)}^2&=\mu
(u^{(\mu)},u^{(\mu)}W)_{L_2(\G_\e)}+\int\limits_{\g_\e}
|u^{(\mu)}|^2\frac{\p W}{\p x_2}\di x_1
\\
&-\mu\|u^{(\mu)}W\|_{L_2(\G_-)}^2 -\RE (W\D u^{(\mu)},W
u^{(\mu)})_{L_2(\Om)}
\\
&-2\RE \left(W\nabla u^{(\mu)},u^{(\mu)}\nabla W\right)_{L_2(\Om)}
-\left(u^{(\mu)}\D W, u^{(\mu)}W\right)_{L_2(\Om)}.
\end{aligned}
\end{equation}
Let us calculate the fifth term in the right hand side of the last
equation. We integrate by parts employing (\ref{1.4}),
\begin{align*}
2\RE \left(W\nabla u^{(\mu)}, u^{(\mu)}\nabla
W\right)_{L_2(\Om)}=&\frac{1}{2}\int\limits_{\Om} \nabla W^2\cdot
\nabla |u^{(\mu)}|^2\di x
\\
=&-\frac{1}{2}\int\limits_{\G_-} W^2\frac{\p}{\p x_2}
|u^{(\mu)}|^2\di x_1 - \frac{1}{2}\int\limits_{\Om} W^2\D
|u^{(\mu)}|^2\di x
\\
=&-\mu\|u^{(\mu)}W\|_{L_2(\G_-)}^2-\RE(W u^{(\mu)},W\D
u^{(\mu)})_{L_2(\Om)}
\\
&-\|W\nabla u^{(\mu)}\|_{L_2(\Om)}^2.
\end{align*}
We substitute the last identity into (\ref{3.36}),
\begin{equation}\label{2.39}
\begin{aligned}
\|\nabla (u^{(\mu)}W)\|_{L_2(\Om)}^2=&\mu
(u^{(\mu)},u^{(\mu)}W)_{L_2(\G_\e)}+ \int\limits_{\g_\e}
|u^{(\mu)}|^2\frac{\p W}{\p x_2}\di x_1
\\
&+ \|W\nabla u^{(\mu)}\|_{L_2(\Om)}^2- \left(u^{(\mu)}\D W,
u^{(\mu)}W\right)_{L_2(\Om)}.
\end{aligned}
\end{equation}
Taking $\d=\e\mu$ in (\ref{2.37a}), we get
\begin{equation}\label{2.40}
\|W\nabla u^{(\mu)}\|_{L_2(\Om)} \leqslant C\e \mu |\ln\e\mu|
\|f\|_{L_2(\Om)}.
\end{equation}
It follows from (\ref{2.36a}) with $\d=\e\mu$ and (\ref{2.33}) that
\begin{equation}\label{2.42}
\big|(u^{(\mu)}\D W,u^{(\mu)}W)_{L_2(\Om)}\big|\leqslant
C\eta^\a\e^{-1/2}\mu^{3/2}|\ln\e\mu| \|f\|_{L_2(\Om)}^2,\quad
\a\in(1/2,1).
\end{equation}
Employing (\ref{2.19}), (\ref{2.17}), by direct calculations we
check that
\begin{align*}
\int\limits_{\g_\e} |u^{(\mu)}|^2\frac{\p W}{\p x_2}\di x_1
=&\sum\limits_{j=-\infty}^{+\infty} \int\limits_{\g_{\e,j}}
|u^{(\mu)}|^2\frac{\p W}{\p x_2}\di x_1
\\
=&\e\mu\sum\limits_{j=-\infty}^{+\infty} \int\limits_{\g_{\e,j}}
|u^{(\mu)}|^2 \frac{\p}{\p x_1} \arcsin \frac{x_1-\e\pi
j}{\e\eta}\di x_1,
\end{align*}
and
\begin{align*}
& \int\limits_{\g_{\e,j}} |u^{(\mu)}|^2 \frac{\p}{\p x_1} \arcsin
\frac{x_1-\e\pi j}{\e\eta}\di x_1= \int\limits_{\e\pi
j-\e\eta}^{\e\pi j} |u^{(\mu)}|^2 \frac{\p}{\p x_1}\left( \arcsin
\frac{x_1-\e\pi j}{\e\eta}+\frac{\pi}{2}\right)\di x_1
\\
&\hphantom{\int\limits_{\g_\e^j} |u^{(\mu)}|^2 = } +
\int\limits_{\e\pi j}^{\e\pi j+\e\eta} |u^{(\mu)}|^2 \frac{\p}{\p
x_1}\left( \arcsin \frac{x_1-\e\pi
j}{\e\eta}-\frac{\pi}{2}\right)\di x_1,
\\
&\hphantom{\int\limits_{\g_\e^j} |u^{(\mu)}|^2 } =\pi
|u^{(\mu)}(\e\pi j,0)|^2+ \int\limits_{\g_{\e,j}} \left( \arcsin
\frac{x_1-\e\pi j}{\e\eta}-\frac{\pi}{2} \sgn(x_1-\e\pi j)\right)
\frac{\p}{\p x_1} |u^{(\mu)}|^2\di x_1,
\end{align*}
where
\begin{equation*}
\pi|u^{(\mu)}(\e\pi j,0)|^2=\frac{1}{\e}\int\limits_{\e\pi
(j-1)}^{\e\pi j} \frac{\p}{\p x_1} \big( (x_1-\e\pi
(j-1))|u^{(\mu)}|^2\big)\di x_1.
\end{equation*}
Thus, in view of the embedding of $\H^2(\Om)$ into $\H^1(\G_-)$ and
(\ref{2.30})
\begin{align*}
\left| \int\limits_{\g_\e} |u^{(\mu)}|^2\frac{\p W}{\p x_2}\di x_1
\right| \leqslant & \mu\sum\limits_{j=-\infty}^{+\infty}
\int\limits_{\e\pi(j-1)}^{\e\pi j} \left|\frac{\p}{\p x_1}
(x_1-\e\pi(j-1)) |u^{(\mu)}|^2\right|\di x_1
\\
&+\e\mu\pi\sum\limits_{j=-\infty}^{+\infty} \int\limits_{\g_\e^j}
\left|\frac{\p}{\p x_1}|u^{(\mu)}|^2\right|\di x_1\leqslant
C\mu\|f\|_{L_2(\Om)}^2.
\end{align*}
We substitute the obtained estimate, (\ref{2.35a}) with
$\widetilde{\d}=\e^2\mu^2$, (\ref{2.40}), (\ref{2.42}) into
(\ref{2.39}) and arrive at (\ref{2.35}).
\end{proof}

The proven lemma and (\ref{1.9}), (\ref{2.36a}) with $\d=\e\mu$
imply
\begin{equation}
\|(\mathcal{H}_\e-\iu)^{-1}-(\mathcal{H}^{(\mu)}-\iu)^{-1}
\|_{L_2(\Om)\to\H^1(\Om)}\leqslant C_1 \mu^{1/2}. \label{1.7}
\end{equation}
The resolvent $(\mathcal{H}^{(\mu)}-\iu)^{-1}$ is obviously analytic
in $\mu$ and thus
\begin{equation*}
\|(\mathcal{H}^{(\mu)}-\iu)^{-1}-(\mathcal{H}^{(0)}-\iu)^{-1}
\|_{L_2(\Om)\to\H^1(\Om)}\leqslant C\mu.
\end{equation*}
This inequality, (\ref{1.7}), and (\ref{1.10}) yield (\ref{1.8}),
(\ref{1.8a}).

\section{Uniform resolvent convergence for $\Hpe(\tau)$}

This section is devoted to the proof of
Theorems~\ref{th1.3},~\ref{th1.4}. The proof of the first theorem is
close in spirit to that of Theorem~2.3 in \cite{BC}. The difference
is that here we employ the corrector $W$ as we did in the previous
section. This is why an essential modification of the proof of
Theorem~2.3 in \cite{BC} is needed.

We begin with several auxiliary lemmas. The first one was proved in
\cite{BC}, see Lemma~4.2 in this paper.

\begin{lem}\label{lm3.1}
Let $|\tau|<1-\k$, where $0<\k<1$, and
\begin{equation*}
U_\e=\left(\Hpe(\tau)-\frac{\tau^2}{\e^2}\right)^{-1}f,\quad f\in
L_2(\Om_\e).
\end{equation*}
Then
\begin{align}
&\|U_\e\|_{L_2(\Om_\e)}\leqslant 4\|f\|_{L_2(\Om_\e)},\label{3.1}
\\
&\Big\|\frac{\p U_\e}{\p x_2}\Big\|_{L_2(\Om_\e)}\leqslant
2\|f\|_{L_2(\Om_\e)}, \nonumber
\\
&\Big\|\frac{\p U_\e}{\p x_1}\Big\|_{L_2(\Om_\e)}\leqslant
\frac{2}{\k^{1/2}}\|f\|_{L_2(\Om_\e)}. \nonumber
\end{align}
If, in addition, $f\in \mathfrak{L}_\e^\bot$, then
\begin{equation}\label{3.4}
\|U_\e\|_{L_2(\Om_\e)}\leqslant
\frac{\e}{\k^{1/2}}\|f\|_{L_2(\Om_\e)},\quad \|\nabla
U_\e\|_{L_2(\Om_\e)}\leqslant \frac{\e}{2\k} \|f\|_{L_2(\Om_\e)}.
\end{equation}
\end{lem}

It was also shown in \cite{BC} in the proof of the last lemma that
for any $u\in\Hoper^1(\Om_\e,\Gp_+)$ and $|\tau|\leqslant 1-\k$
\begin{equation}\label{3.12}
\begin{aligned}
&\Big\| \left(\iu \frac{\p}{\p x_1}-
\frac{\tau}{\e}\right)u\Big\|_{L_2(\Om_\e)}^2-\frac{\tau^2}{\e^2}
\|u\|_{L_2(\Om_\e)}^2\geqslant \k \Big\|\frac{\p u}{\p
x_1}\Big\|_{L_2(\Om_\e)}^2,
\\
&\Big\|\frac{\p u}{\p x_2}\Big\|_{L_2(\Om_\e)}\geqslant
\frac{1}{2}\|u\|_{L_2(\Om)}.
\end{aligned}
\end{equation}

\begin{lem}\label{lm3.2}
Let $F\in L_2(0,\pi)$. Then
\begin{equation*}
|(\mathcal{Q}_\mu^{-1}F)(0)|\leqslant 5\|F\|_{L_2(0,\pi)}.
\end{equation*}
\end{lem}

\begin{proof}
We can find $\mathcal{Q}_\mu^{-1}F$ explicitly
\begin{equation*}
(\mathcal{Q}_\mu^{-1}F)(x_2)=-\frac{1}{2}\int\limits_{0}^{\pi}\left(
|x_2-t|-\pi+ \frac{x_2-\pi}{1+\pi\mu}(1+\mu(t-\pi))\right)F(t)\di t.
\end{equation*}
Hence, by the Cauchy-Schwartz inequality
\begin{equation*}
|(\mathcal{Q}_\mu^{-1}F)(0)|\leqslant \frac{1}{2(1+\pi
\mu)}\int\limits_{0}^{\pi} (2\pi-t)|F(t)|\di t\leqslant
5\|F\|_{L_2(0,\pi)},
\end{equation*}
that completes the proof.
\end{proof}

\begin{proof}[Proof of Theorem~\ref{th1.3}]
Let $f\in L_2(\Om_\e)$, $f=F_\e+f_\e^\bot$, where $F_\e\in
\mathfrak{L}_\e$, $f_\e^\bot\in\mathfrak{L}_\e^\bot$,
\begin{gather}
F_\e(x_2)=\frac{1}{\e\pi}\int\limits_{-\frac{\e\pi}{2}}^{\frac{\e\pi}{2}}
f_\e(x)\di x_1,\nonumber
\\
\e\pi\|F_\e\|_{L_2(0,\pi)}^2+\|f_\e^\bot\|_{L_2(\Om_\e)}^2=
\|f\|_{L_2(\Om_\e)}^2. \label{3.5}
\end{gather}
Then
\begin{equation*}
\left(\Hpe(\tau)-\frac{\tau^2}{\e^2}\right)^{-1}f=
\left(\Hpe(\tau)-\frac{\tau^2}{\e^2}\right)^{-1}F_\e+
\left(\Hpe(\tau)-\frac{\tau^2}{\e^2}\right)^{-1}f_\e^\bot.
\end{equation*}
By (\ref{3.4}), (\ref{3.5}) we obtain
\begin{equation}\label{3.6}
\bigg\| \left(\Hpe(\tau)-\frac{\tau^2}{\e^2}\right)^{-1}f_\e^\bot
\bigg\|_{L_2(\Om_\e)}\leqslant
\frac{\e}{\k^{1/2}}\|f_\e^\bot\|_{L_2(\Om_\e)}\leqslant
\frac{\e}{\k^{1/2}}\|f\|_{L_2(\Om_\e)}.
\end{equation}

We denote
\begin{align*}
&U_\e:=\left(\Hpe(\tau)-\frac{\tau^2}{\e^2}\right)^{-1}F_\e,\quad
U^{(\mu)}_\e:=\mathcal{Q}_\mu^{-1}F_\e,
\\
&V_\e(x):=U_\e(x)-U^{(\mu)}_\e(x)-U^{(\mu)}_\e(0)W(x,\e,\mu)\chi_1(x_2),
\end{align*}
where, we remind, the function $\chi_1$ was introduced in the third
section. In view of (\ref{2.1}) and the definition of $U_\e$ the
function $V_\e$ belongs to $\Hoper^1(\Om_\e,\Gp_+\cup\gp_\e)$.

We write the integral identities for $U_\e$ and $U^{(\mu)}_\e$,
\begin{equation}
\hpe(\tau)[U_\e,\phi]-\frac{\tau^2}{\e^2}(U_\e,\phi)_{L_2(\Om_\e)}
=(F_\e,\phi)_{L_2(\Om_\e)}\label{3.7}
\end{equation}
for all $\phi\in\Hoper^1(\Om_\e,\Gp_+\cup\gp_\e)$, and
\begin{equation}\label{3.8}
\left(\frac{d U^{(\mu)}_\e}{dx_2},\frac{d\phi}{d
x_2}\right)_{L_2(0,\pi)}+\mu U^{(\mu)}_\e(0)\overline{\phi(0)}
=(F_\e,\phi)_{L_2(0,\pi)}
\end{equation}
for all $\phi\in\Ho^1((0,\pi),\{\pi\})$. Given any
$\phi\in\Hoper^1(\Om_\e,\Gp_+)$, for a.e. $x_1\in(-\e\pi/2,\e\pi/2)$
we have $\phi(x_1,\cdot)\in\Ho^1((0,\pi),\{\pi\})$. We take such
$\phi$ in (\ref{3.8}) and integrate it over
$x_1\in(-\e\pi/2,\e\pi/2)$,
\begin{equation*}
\left(\frac{d U^{(\mu)}_\e}{dx_2},\frac{\p\phi}{\p
x_2}\right)_{L_2(\Om_\e)}+\mu (U^{(\mu)}_\e, \phi)_{L_2(\Gp_-)}
=(F_\e,\phi)_{L_2(\Om_\e)}.
\end{equation*}
The function $U^{(\mu)}_\e$ is independent of $x_1$, and hence
\begin{align*}
\left( \left(\iu\frac{\p}{\p
x_1}-\frac{\tau}{\e}\right)U^{(\mu)}_\e, \left(\iu\frac{\p}{\p
x_1}-\frac{\tau}{\e}\right)\phi\right)_{L_2(\Om_\e)}=&
-\frac{\tau}{\e}\left(U^{(\mu)}_\e, \left(\iu\frac{\p}{\p
x_1}-\frac{\tau}{\e}\right)\phi\right)_{L_2(\Om_\e)}
\\
=&\frac{\tau^2}{\e^2} (U^{(\mu)}_\e,\phi)_{L_2(\Om_\e)}.
\end{align*}
The sum of two last equations is as follows,
\begin{equation}\label{3.9}
\begin{aligned}
&\left( \left(\iu\frac{\p}{\p
x_1}-\frac{\tau}{\e}\right)U^{(\mu)}_\e, \left(\iu\frac{\p}{\p
x_1}-\frac{\tau}{\e}\right)\phi\right)_{L_2(\Om_\e)}+ \left(\frac{\p
U^{(\mu)}_\e}{\p x_2},\frac{\p\phi}{\p x_2}\right)_{L_2(\Om_\e)}
\\
&\hphantom{\Bigg(\Bigg(}
-\frac{\tau^2}{\e^2}(U^{(\mu)}_\e,\phi)_{L_2(\Om_\e)} +\mu
(U^{(\mu)}_\e, \phi)_{L_2(\Gp_-)} =(F_\e,\phi)_{L_2(\Om_\e)}
\end{aligned}
\end{equation}
We let $\phi=V_\e$ in (\ref{3.7}), (\ref{3.9}) and take the
difference of these two equations,
\begin{align*}
&\left( \left(\iu\frac{\p}{\p
x_1}-\frac{\tau}{\e}\right)(U_\e-U^{(\mu)}_\e),
\left(\iu\frac{\p}{\p x_1}-\frac{\tau}{\e}\right) V_\e
\right)_{L_2(\Om_\e)}+ \left(\frac{\p}{\p
x_2}(U_\e-U^{(\mu)}_\e),\frac{\p V_\e}{\p x_2}\right)_{L_2(\Om_\e)}
\\
&\hphantom{\left(\iu\frac{\p}{\p
x_1}-\frac{\tau}{\e}\right)(U_\e-U^{(\mu)}_\e), \iu\frac{\p}{\p
x_1}- - - }
-\frac{\tau^2}{\e^2}(U_\e-U^{(\mu)}_\e,V_\e)_{L_2(\Om_\e)}=\mu
(U^{(\mu)}_\e,V_\e)_{L_2(\Gp_-)}.
\end{align*}
We represent $U_\e-U^{(\mu)}_\e$ as $V_\e+U^{(\mu)}_\e(0)W\chi_1$
and substitute it into the last equation,
\begin{equation}\label{3.10}
\begin{aligned}
&\left\|\left(\iu\frac{\p}{\p x_1}-\frac{\tau}{\e}\right)V_\e
\right\|_{L_2(\Om_\e)}^2+\left\|\frac{\p V_\e}{\p
x_2}\right\|_{L_2(\Om_\e)}^2-\frac{\tau^2}{\e^2}\|V_\e\|_{L_2(\Om_\e)}^2
\\
&=\mu(U^{(\mu)}_\e,V_\e)_{L_2(\Gp_\e)}-U^{(\mu)}_\e(0) \left(
\Big(\iu\frac{\p}{\p x_1}-\frac{\tau}{\e}\Big) W\chi_1,
\left(\iu\frac{\p}{\p x_1}-\frac{\tau}{\e}\right) V_\e
\right)_{L_2(\Om_\e)}
\\
&\hphantom{=}-U^{(\mu)}_\e(0) \left(\frac{\p W\chi_1}{\p
x_2},\frac{\p V_\e}{\p x_2}\right)_{L_2(\Om_\e)}
-\frac{\tau^2}{\e^2}U^{(\mu)}_\e(0)(W\chi_1,V_\e)_{L_2(\Om_\e)}
\\
&=U^{(\mu)}_\e(0)\left(\mu (W,V_\e)_{L_2(\Gp_\e)}-(\nabla
W\chi_1,\nabla V_\e)_{L_2(\Om_\e)} -\frac{2\iu\tau}{\e}
\left(\frac{\p W\chi_1}{\p x_1},V_\e\right)_{L_2(\Om_\e)}\right).
\end{aligned}
\end{equation}
We integrate by parts employing (\ref{2.1}),
\begin{equation*}
-\frac{2\iu\tau}{\e} \left(\frac{\p W\chi_1}{\p
x_1},V_\e\right)_{L_2(\Om_\e)}= \frac{2\iu\tau}{\e}
\left(W,\chi_1\frac{\p V_\e}{\p x_1}\right)_{L_2(\Om_\e)},
\end{equation*}
and
\begin{align*}
&\mu(W,V_\e)_{L_2(\Gp_\e)}-(\nabla (W\chi_1),\nabla
V_\e)_{L_2(\Om_\e)}
\\
&\hphantom{\mu(W,V_\e)_{L_2(\Gp_\e)}}=
\mu(W,V_\e)_{L_2(\Gp_\e)}+\left(\frac{\p W}{\p
x_2},V_\e\right)_{L_2(\Gp_\e)}+(\D (W\chi_1), V_\e)_{L_2(\Om_\e)}
\\
&\hphantom{\mu(W,V_\e)_{L_2(\Gp_\e)}}=(\D W\chi_1,
V_\e)_{L_2(\Om_\e)}.
\end{align*}
Together with (\ref{3.10}) it yields
\begin{equation}\label{3.11}
\begin{aligned}
&\left\|\left(\iu\frac{\p}{\p x_1}-\frac{\tau}{\e}\right)V_\e
\right\|_{L_2(\Om_\e)}^2+\left\|\frac{\p V_\e}{\p
x_2}\right\|_{L_2(\Om_\e)}^2-\frac{\tau^2}{\e^2}\|V_\e\|_{L_2(\Om_\e)}^2
\\
&\hphantom{\Bigg\|\Bigg\|}=U^{(\mu)}_\e(0)\left( \big(\D (W\chi_1),
V_\e\big)_{L_2(\Om_\e)}+ \frac{2\iu\tau}{\e} \left(W\chi_1,\frac{\p
V_\e}{\p x_1}\right)_{L_2(\Om_\e)} \right).
\end{aligned}
\end{equation}
It follows from Lemma~\ref{lm3.2} and (\ref{3.5}) that
\begin{equation*}
|U^{(\mu)}_\e(0)|\leqslant 5\pi\e^{-1/2} \|f\|_{L_2(\Om_\e)}.
\end{equation*}
Hence, we can estimate the right hand side of (\ref{3.11}) as
follows,
\begin{align*}
&\left| U^{(\mu)}_\e(0)\left(\big(\D (W\chi_1), V_\e
\big)_{L_2(\Om_\e)}+\frac{2\iu \tau}{\e} \left(W\chi_1,\frac{\p
V_\e}{\p x_1}\right)_{L_2(\Om_\e)}\right)\right|
\\
&\leqslant 5\pi\e^{-1/2}\|f\|_{L_2(\Om_\e)} \left( \|\D
(W\chi_1)\|_{L_2(\Om_\e)}\|V_\e\|_{L_2(\Om_\e)}+2\e^{-1}\|W\chi_1\|_{L_2(\Om_\e)}
\left\|\frac{\p V_\e}{\p x_1}\right\|_{L_2(\Om_\e)}\right)
\\
&\leqslant 50\pi^2\e^{-1}\|\D (W\chi_1)\|_{L_2(\Om)}^2
\|f\|_{L_2(\Om_\e)}^2+\frac{1}{8}\|V_\e\|_{L_2(\Om_\e)}^2
\\
&\hphantom{\leqslant}+ 25\pi^2
\k^{-1}\e^{-3}\|W\|_{L_2(\Om_\e)}^2\|f\|_{L_2(\Om_\e)}^2+\k\left\|
\frac{\p V_\e}{\p x_1}\right\|_{L_2(\Om_\e)}^2.
\end{align*}
We substitute this inequality and (\ref{3.12}) into (\ref{3.11}),
\begin{align*}
\k&\left\|\frac{\p V_\e}{\p
x_1}\right\|_{L_2(\Om_\e)}^2+\frac{1}{4}\|V_\e\|_{L_2(\Om_\e)}^2
\leqslant 50\pi^2\e^{-1}\|f\|_{L_2(\Om_\e)}^2\|\D
(W\chi_1)\|_{L_2(\Om_\e)}^2
\\
&+25\pi^2\k^{-1}\e^{-3}\|W\|_{L_2(\Om_\e)}^2\|f\|_{L_2(\Om_\e)}^2
+\frac{1}{8}\|V_\e\|_{L_2(\Om_\e)}^2 +\k \left\|\frac{\p V_\e}{\p
x_1}\right\|_{L_2(\Om_\e)}^2,
\\
&\|V_\e\|_{L_2(\Om_\e)}^2\leqslant
C\left(\e^{-1}\|f\|_{L_2(\Om_\e)}^2\|\D (W\chi_1)\|_{L_2(\Om_\e)}^2
+\k^{-1}\e^{-3}\|f\|_{L_2(\Om_\e)}^2\|W\|_{L_2(\Om_\e)}^2\right),
\\
&\|V_\e\|_{L_2(\Om_\e)}\leqslant C\left(\e^{-1/2}\|\D
(W\chi_1)\|_{L_2(\Om_\e)} + \k^{-1/2}\e^{-3/2}
\|W\|_{L_2(\Om_\e)}\right) \|f\|_{L_2(\Om_\e)},
\end{align*}
where the constants $C$ are independent of $\e$, $\mu$, $\k$, and
$f$. Combining the last inequality, (\ref{3.5}) and
Lemma~\ref{lm3.2}, we arrive at
\begin{equation}\label{3.13a}
\begin{aligned}
\|U_\e-U^{(\mu)}_\e&\|_{L_2(\Om_\e)}\leqslant \|V_\e\|_{L_2(\Om_\e)}
+|U^{(\mu)}(0)| \|W\|_{L_2(\Om_\e)}
\\
\leqslant & \|V_\e\|_{L_2(\Om_\e)}
+C\e^{-1/2}\|f\|_{L_2(\Om_\e)}\|W\|_{L_2(\Om_\e)}
\\
\leqslant &C\left(\e^{-1/2}\|\D W\chi_1\|_{L_2(\Om_\e)} +
\k^{-1/2}\e^{-3/2}\|W\|_{L_2(\Om_\e)}\right) \|f\|_{L_2(\Om_\e)},
\end{aligned}
\end{equation}
where the constants $C$ are independent of $\e$, $\mu$, $\k$, and
$f$.

Let us estimate $\|W\|_{L_2(\Om_\e)}$ and $\|\D
(W\chi_1)\|_{L_2(\Om_\e)}$. We have
\begin{equation*}
\|W\|_{L_2(\Om_\e)}^2=\|W\|_{L_2(\Om_\e\setminus\Om^\d)}^2+
\|W\|_{L_2(\Om_\e\cap\Om^\d)}^2.
\end{equation*}
We take $\d=\frac{3}{2}\eta^\a$ and in view of the definition
(\ref{2.19}) of $W$ we obtain
\begin{align*}
\|W\|_{L_2(\Om_\e\setminus\Om^\d)}^2=\e^2\mu^2
\|X\|_{L_2(\Om_\e\setminus\Om^\d)}^2\leqslant
\e^4\mu^2\int\limits_{|\xi_1|<\frac{\pi}{2},\,\xi_2>0}
|X(\xi)|^2\di\xi \leqslant C\e^4\mu^2,
\end{align*}
where the constant $C$ is independent of $\e$, $\mu$, $\k$, and $f$.
It follows from (\ref{2.28}) that
\begin{equation*}
\|W\|_{L_2(\Om_\e\cap\Om^{\frac{3}{2}\eta^\a})}^2\leqslant
C\e^2\eta^{2\a},\quad \a\in(0,1),
\end{equation*}
where the constant $C$ is independent of $\e$ and $\eta$. Hence,
\begin{equation}\label{3.15}
\|W\|_{L_2(\Om_\e)}\leqslant C\e^2\mu,
\end{equation}
where the constant $C$ is independent of $\e$ and $\mu$.

The definition (\ref{2.19}) of $W$, the equations in (\ref{2.11}),
(\ref{2.16}), the estimate (\ref{2.26}), and the exponential decay
of $X$,
\begin{equation*}
X(\xi)=\Odr(\E^{-2\xi_1}),\quad \xi_2\to+\infty
\end{equation*}
yield that
\begin{align*}
&\|\D (W\chi_1)\|_{L_2(\Om_\e)}^2\leqslant 2\|\D
W\|_{L_2(\Om_\e)}^2+ 2\left\| 2\frac{\p W}{\p
x_1}\chi_1'+W\chi_1''\right\|_{L_2(\Om_\e)}^2,
\\
&\|\D W\|_{L_2(\Om_\e)}^2\leqslant C\mu^2\eta^{2-2\a},\quad
\a\in(1/2,1),
\\
& \left\| 2\frac{\p W}{\p
x_1}\chi_1'+W\chi_1''\right\|_{L_2(\Om_\e)}^2 \leqslant
C\mu^2\E^{-2\e^{-1}},
\end{align*}
where $C$ are positive constants independent of $\e$, $\eta$, and
$\mu$. We substitute the last estimates and (\ref{3.15}) into
(\ref{3.13a}),
\begin{equation*}
\|U_\e-U^{(\mu)}_\e\|_{L_2(\Om_\e)}\leqslant C \k^{-1/2}\mu\e^{1/2}
\|f\|_{L_2(\Om_\e)},
\end{equation*}
where the constant $C$ is independent of $\e$, $\mu$, and $\k$.
Together with (\ref{3.6}) it completes the proof.
\end{proof}

\begin{proof}[Proof of Theorem~\ref{th1.4}]
First we obtain the upper bound for the eigenvalues $\l_n$. To do
this, we employ standard bracketing arguments (see, for instance,
\cite[Ch. X\!I\!I\!I, Sec. 15, Prop. 4]{RS4}), and estimate the
eigenvalues of $\Hpe(\tau)$ by those of the same operator but with
$\eta=\pi/2$, i.e., with Dirichlet boundary condition on $\Gp_-$.
The lowest eigenvalues of the latter operator are
\begin{equation*}
\frac{\tau^2}{\e^2}+n^2,\quad
\frac{(2+\tau)^2-\tau^2}{\e^2}+n^2,\quad
\frac{(2-\tau)^2-\tau^2}{\e^2}+n^2,\quad n=1,2,\ldots
\end{equation*}
Hence, for $n^2<4\k\e^{-2}$ the lowest eigenvalues among mentioned
are $\tau^2\e^{-2}+n^2$, and thus
\begin{equation}\label{3.16}
\frac{1}{4}\leqslant \l_n(\tau,\e)-\frac{\tau^2}{\e^2}\leqslant
n^2,\quad n< 2\k^{1/2}\e^{-1}.
\end{equation}
The lower estimate was obtained by replacing the boundary conditions
on $\Gp_-$ by the Neumann one. In the same way we can estimate the
eigenvalues of $Q_\mu$ replacing the boundary condition at $x_2=0$
by the Dirichlet and Neumann one,
\begin{equation}\label{3.17}
0\leqslant \L_n(\mu)\leqslant n^2
\end{equation}
uniformly in $\mu$ for all $n\in \mathds{Z}$.

By \cite[Ch. I\!I\!I, Sec. 1, Th. 1.4]{OIS}, Theorem~\ref{th1.3},
and (\ref{3.16}), (\ref{3.17}) we get
\begin{align*}
\left|\frac{1}{\l_n(\tau,\e)-\frac{\tau^2}{\e^2}}-\frac{1}{\L_n(\mu)}\right|
&\leqslant C\k^{-1/2}\e^{1/2}\mu,
\\
\left|\l_n(\tau,\e)-\frac{\tau^2}{\e^2}-\L_n(\mu)\right|&\leqslant
C\k^{-1/2}(\mu\e^{1/2}+\e)|\L_n(\mu)| \left|\l_n(\tau,\e)-\frac{\tau^2}{\e^2}\right|
\\
&\leqslant C n^4 \k^{-1/2}(\mu\e^{1/2}+\e),
\end{align*}
which proves (\ref{1.15}).

The eigenvalues $\L_n(\mu)$ are solutions to the equation
(\ref{1.17}), and the associated eigenfunctions are $\sin
\sqrt{\L_n}(x_2-\pi)$. Hence, these eigenvalues are holomorphic with
respect to $\mu$ by the inverse function theorem. The formula
(\ref{1.16}) can be checked by expanding the equation (\ref{4.69})
and $\L_n(\mu)$ w.r.t. $\mu$.
\end{proof}

\section{Bottom of the spectrum}

In this section we prove Theorem~\ref{th1.5}. The proof of
(\ref{1.18}) reproduces word by word the proof of similar equation
(2.5) in \cite{BC} with one minor change, namely, one should use
here identity
\begin{equation}\label{4.0}
\l_1(0,\e)=\frac{1}{4}+o(1),\quad \e\to+0,
\end{equation}
instead of similar identity in \cite{BC}. The identity (\ref{4.0})
follows from (\ref{1.15}), (\ref{1.16}).

In order to construct the asymptotic expansion for $\l_1(0,\e)$, we
employ the approach suggested in \cite{AsAn}, \cite{GD98},
\cite{Gzh01}, \cite{GDu99} for studying similar problems in bounded
domains.

The eigenvalue $\l_1(0,\e)$ and the associated eigenfunction
$\po(x,\e)$ of $\Hpe(0)$ satisfy the problem
\begin{equation}\label{4.4}
\begin{aligned}
-&\D\po(x,\e)=\l_1(0,\e)\po(x,\e)\quad\text{in}\quad \Om_\e,
\\
&\po(x,\e)=0\quad\text{on}\quad \Gp_+\cup\gp_\e,\qquad \frac{\p\po
}{\p x_2}(x,\e)=0\quad \text{on}\quad \Gp_\e.
\end{aligned}
\end{equation}
and periodic boundary conditions on the lateral boundaries of
$\Om_\e$. We construct the asymptotics for $\l_1(0,\e)$ as
\begin{equation*}
\l_1(0,\e)=\L(\e,\mu),
\end{equation*}
where $\L=\L(\e,\mu)$ is a function to be determined. It view of
(\ref{1.15}) with $\tau=0$ the function $\L$ should satisfy
(\ref{4.50}).

The asymptotics of the associated eigenfunction $\po_\e$ is
constructed as the sum of three expansion, namely, the external
expansion, the boundary layer, and the internal expansion. The
external expansion has a closed form,
\begin{equation}\label{4.2}
\psi_\e^{ex}(x,\L)=\sin\sqrt{\L}(x_2-\pi).
\end{equation}
It is clear that for any choice of $\L(\e,\mu)$ this function solves
the equation in (\ref{4.4}), and satisfies the periodic boundary
conditions on the lateral boundaries of $\Om_\e$.

The boundary layer is constructed in terms of the variables $\xi$,
i.e., $\psi_\e^{bl}=\psi_\e^{bl}(\xi,\mu)$. The main aim of
introducing the boundary layer is to satisfy the boundary condition
on $\Gp_\e$. We construct $\psi_\e^{bl}$ by the boundary layer
method. In accordance with this method, the series $\psi_\e^{bl}$
should satisfy the equation in (\ref{4.4}), the periodic boundary
condition on the lateral boundaries of $\Om_\e$, the boundary
condition
\begin{equation}\label{4.5}
\frac{\p\psi^{ex}_\e}{\p x_2}+\frac{\p \psi_\e^{bl}}{\p x_2}=0\quad
\text{on}\quad \Gp_\e,
\end{equation}
and it should decay exponentially as $\xi_2\to+\infty$.

It follows from (\ref{4.2}) and the definition of $\xi$ that
$\psi_\e^{bl}$ should satisfy the boundary condition
\begin{gather}
\frac{\p\psi_\e^{bl}}{\p\xi_2}=-\sqrt{\L}\cos\sqrt{\L}\pi
\quad\text{on}\quad \Gp^{0}, \label{4.51}
\\
\Gp^{0}:=\left\{\xi:
0<|\xi_1|<\frac{\pi}{2},\,\xi_2>0\right\}.\nonumber
\end{gather}
Here we passed to the limit $\eta\to+0$ in the definition of
$\Gp_\e$.

We substitute $\psi_\e^{bl}$ into the equation in (\ref{4.4}) and
rewrite it in the variables $\xi$,
\begin{equation}\label{4.52}
-\D_\xi \psi_\e^{bl}=\e^2\L\psi_\e^{bl},\quad \xi\in\Pi,\qquad
\Pi:=\left\{\xi: |\xi_1|<\frac{\pi}{2},\ \xi_2>0\right\}.
\end{equation}
To construct $\psi_\e^{bl}$, in \cite{AsAn}, \cite{GD98},
\cite{Gzh01}, \cite{GDu99} the authors used the standard way.
Namely, they sought $\psi_\e^{bl}$ and $\L(\e,\mu)$ as asymptotic
series power in $\e$. Then these series were substituted into
(\ref{4.51}), (\ref{4.52}), and equating the coefficients at like
powers of $\e$ implied the boundary value problems for the
coefficients of the mentioned series. In our case we do not employ
this way. Instead of this we study the existence of the required
solution to the problem (\ref{4.51}), (\ref{4.52}) and describe some
of its properties needed in what follows.

By $\mathfrak{V}$ we denote the space of $\pi$-periodic even in
$\xi_1$ functions belonging to
$C^\infty(\overline{\Pi}\setminus\{0\})$ and exponentially decaying
as $\xi_2\to+\infty$ together with all their derivatives uniformly
in $\xi_1$. We observe that $X\in \mathfrak{V}$.

\begin{lem}\label{lm4.3}
The function $X$ can be represented as the series
\begin{equation}\label{4.53}
X(\xi)=-\sum\limits_{n=1}^{+\infty}
\frac{1}{n}\E^{-2n\xi_2}\cos2n\xi_1,
\end{equation}
which converges in $L_2(\Pi)$ and in $C^k(\overline{\Pi}\cap\{\xi:
\xi\geqslant R\})$ for each $k\geqslant 0$, $R>0$.
\end{lem}

\begin{proof}
Since $X\in\mathfrak{V}$, for each $\xi_2>0$ and each $k\geqslant 0$
we can expand it in $C^k[-\pi/2,\pi/2]$,
\begin{align}
&X(\xi)=\sum\limits_{n=1}^{+\infty} X_n(\xi_2)\cos 2n\xi_1,\quad
\|X(\cdot,\xi_2)\|_{L_2\left(-\frac{\pi}{2},\frac{\pi}{2}\right)}^2=
\frac{\pi}{2}\sum\limits_{n=1}^{+\infty} X_n^2(\xi_2), \label{4.54}
\\
&X_n(\xi_2)=\frac{2}{\pi}\int\limits_{-\frac{\pi}{2}}^{\frac{\pi}{2}}
X(\xi)\cos 2n\xi_1\di\xi_1.\nonumber
\end{align}
Integrating the second equation in (\ref{4.54}) w.r.t. $\xi_2$, we
obtain the Parseval identity
\begin{equation*}
\|X\|_{L_2(\Pi)}^2= \frac{\pi}{2}\sum\limits_{n=1}^{+\infty}
\|X_n\|_{L_2(0,+\infty)}^2.
\end{equation*}
It yields that the first series in (\ref{4.54}) converges also in
$L_2(\Pi)$, since
\begin{equation*}
\Big\|X-\sum\limits_{n=1}^{N} X_n\cos 2n\xi_1\Big\|_{L_2(\Pi)}^2=
\|X\|_{L_2(\Pi)}^2-\frac{\pi}{2}
\sum\limits_{n=1}^{N}\|X_n\|_{L_2(0,+\infty)}^2.
\end{equation*}
The harmonicity of $X$ and the exponential decay as
$\xi_2\to+\infty$ yield
\begin{align*}
&X_n''(\xi_2)=-\int\limits_{-\frac{\pi}{2}}^{\frac{\pi}{2}}\frac{\p^2
X}{\p\xi_1^2}\cos 2n\xi_1\di \xi_1=-n^2 X_n(\xi_2),
\\
&X_n(\xi_2)=k_n\E^{-2n\xi_2},\quad
k_n=\frac{2}{\pi}\int\limits_{\Gp^0} X_n\cos 2n\xi_1\di\xi_1.
\end{align*}

Denote $\Pi_\d:=\Pi\setminus\{\xi: |\xi|<\d\}$. Employing
(\ref{2.11}) and the harmonicity of $X$, we integrate by parts,
\begin{equation}\label{4.55a}
\begin{aligned}
0=&-\lim\limits_{\d\to+0}\int\limits_{\Pi} \E^{-2n\xi_2}\cos 2n\xi_1
\D_\xi X\di\xi
\\
&=\int\limits_{\Gp^0} \left(\cos 2n\xi_1\frac{\p X}{\p\xi_2}+2n
X\cos 2n\xi_1\right) \di\xi_1
\\
&+\lim\limits_{\d\to+0} \int\limits_{|\xi|<\d,\,\xi_2>0}\left(
\E^{-2n\xi_2} \cos2n\xi_1\frac{\p
X}{\p|\xi|}-X\frac{\p}{\p|\xi|}\E^{-2n\xi_2}\cos 2n\xi_1\right)\di
s
\\
=&-\int\limits_{\Gp^0} \cos 2n\xi_1\di \xi_1+\pi n k_n+\pi.
\end{aligned}
\end{equation}
Thus, $k_n=-1/n$, which implies (\ref{4.53}). The convergence of
this series in $C^k(\overline{\Pi}\cap\{\xi: \xi_2\geqslant R\})$
follows from the exponential decay of its terms in (\ref{4.52}) as
$n\to+\infty$.
\end{proof}

\begin{lem}\label{lm4.4}
For small real $\b$ the problem
\begin{equation}\label{4.56}
-\D_\xi Z-\b^2 Z=\b^2 X,\quad \xi\in\Pi,\qquad \frac{\p
Z}{\p\xi_2}=0,\quad \xi\in\Gp^0,
\end{equation}
has a solution in $\H^2(\Pi)\cap \mathfrak{V}$. This solution and
all its derivatives w.r.t. $\xi$ decay exponentially as
$\xi_2\to+\infty$ uniformly in $\xi_1$ and $\b$. The differentiable
asymptotics
\begin{equation}\label{4.57}
Z(\xi,\b)=Z(0,\b)+\Odr(|\xi|^2\ln|\xi|),\quad\xi\to0,
\end{equation}
holds true uniformly in $\b$. The function $(X+Z)$ is bounded in
$L_2(\Pi)$ uniformly in $\b$. The identity
\begin{equation}\label{4.58}
Z(0,\b)=\b^2\tht(\b^2)
\end{equation}
is valid, where the function $\tht
$ is defined in (\ref{1.22}).
The function $\tht
$ is holomorphic 
 and its Taylor series
is (\ref{4.58a}).
\end{lem}

\begin{proof}
Let $\mathfrak{W}$ be the subspace of $\H^2(\Pi)$ consisting of the
functions satisfying periodic boundary conditions on the lateral
boundaries of $\Pi$, the Neumann boundary condition on $\Gp^0$, and
being orthogonal in $L_2(\Pi)$ to all functions $\phi=\phi(\xi_2)$
belonging to $L_2(\Pi)$. The space $\mathfrak{W}$ is the Hilbert
one.

By $\mathcal{B}$ we denote the operator in $L_2(\Pi)$ acting as
$-\D_\xi$ on $\mathfrak{W}$. This operator is symmetric and closed.
It follows from the definition of $\mathfrak{W}$ that each
$v\in\mathfrak{W}$ satisfies the equation
\begin{equation*}
\int\limits_{-\frac{\pi}{2}}^{\frac{\pi}{2}}
v(\xi)\di\xi_1=0\quad\text{for a.e.}\ \xi_2\in(0,+\infty).
\end{equation*}
Using this fact, one can check easily that $\mathcal{B}\geqslant 4$,
and therefore the bounded inverse operator exists, and
$\|\mathcal{B}^{-1}\|\leqslant 1/4$. Hence,
\begin{equation*}
(\mathcal{B}-\b^2)^{-1}=\mathcal{B}^{-1}(\I-\b^2
\mathcal{B}^{-1})^{-1},
\end{equation*}
i.e., the inverse operator $(\mathcal{B}-\b^2)^{-1}$ exists and is
bounded uniformly in $\b$.

We let $Z:=\b^2(\mathcal{B}-\b^2)^{-1} X$. It is clear that the
function $Z\in\H^2(\Pi)$ solves (\ref{4.56}) and satisfies the
periodic boundary conditions on the lateral boundaries of $\Pi$. By
the standard smoothness improving theorems and the smoothness of $X$
we conclude that $Z\in C^\infty(\overline{\Pi}\setminus\{0\})$.

Using Lemma~\ref{lm4.3}, for $\xi_2>0$ we can also construct $Z$ by
the separation of variables,
\begin{equation}\label{4.59}
Z(\xi,\b)=\sum\limits_{n=1}^{+\infty} \frac{1}{n}\left(\E^{-2n\xi_2}
- \frac{2n}{\sqrt{4n^2-\b^2}}\E^{-\sqrt{4n^2-\b^2}\xi_2}\right) \cos
2n\xi_1.
\end{equation}
In the same way as in the proof of Lemma~\ref{lm4.3} one can check
that this series converges in $L_2(\Pi)$ and
$C^k(\overline{\Pi}\cap\{\xi: \xi_2\geqslant R\})$ for each
$k\geqslant 0$, $R>0$. Thus, this function and all its derivatives
w.r.t. $\xi$ decay exponentially as $\xi_2\to+\infty$ uniformly in
$\xi_1$ and $\b$, and $Z\in \mathfrak{V}$.

By (\ref{4.53}), (\ref{4.59}) we have
\begin{align*}
&X+Z=-\sum\limits_{n=1}^{+\infty}\frac{2}{\sqrt{4n^2-\b^2}}
\E^{-\sqrt{4n^2-\b^2}\xi_2}\cos 2n\xi_1,
\\
&\|X+Z\|_{L_2(\Pi)}^2=\sum\limits_{n=1}^{+\infty}
\frac{\pi}{4n^2-\b^2}\int\limits_{n=1}^{+\infty}
\E^{-2\sqrt{4n^2-\b^2}\xi_2}\di\xi_2= \sum\limits_{n=1}^{+\infty}
\frac{\pi}{2(4n^2-\b^2)^{3/2}}.
\end{align*}
Hence, the function $(X+Z)$ is bounded in $L_2(\Pi)$ uniformly in
$\b$.

Reproducing the proof of Lemma~3.2 in \cite{G92}, one can show
easily that the function $Z$ satisfies differentiable asymptotics
(\ref{4.57}) uniformly in $\b$. Let us calculate $Z(0,\b)$. The
function
\begin{equation}\label{4.60}
\widetilde{Z}(\xi,\b):=X(\xi)+Z(\xi,\b)+\b^{-1}\sin\b\xi_2
\end{equation}
solves the boundary value problem
\begin{equation*}
(\D_\xi+\b^2)\widetilde{Z}=0,\quad\xi\in\Pi,\qquad \frac{\p
\widetilde{Z}}{\p\xi_2}=0,\quad \xi\in\Gp^{0},
\end{equation*}
is bounded, satisfies periodic boundary condition on the lateral
boundaries of $\Pi$, and has the asymptotics
\begin{equation*}
\widetilde{Z}(\xi,\b)=\ln|\xi|+\Odr(1),\quad \xi\to0.
\end{equation*}
Using these properties and (\ref{4.56}), we integrate by parts in
the same way as in (\ref{4.55a}),
\begin{align*}
\b^2\int\limits_{\Pi} X \widetilde{Z}\di\xi=&-\lim\limits_{\d\to+0}
\int\limits_{\Pi_\d} \widetilde{Z} (\D_\xi+\b^2)Z\di \xi
\\
=& \lim\limits_{\d\to+0}\int\limits_{|\xi|=\d,\,\xi_2>0} \left(
\widetilde{Z}\frac{\p Z}{\p|\xi|}-Z\frac{\p \widetilde{Z}}{\p|\xi|}
\right)\di s=-\pi Z(0,\b),
\end{align*}
and hence
\begin{equation*}
Z(0,\b)=-\frac{\b^2}{\pi}\int\limits_{\Pi} X\widetilde{Z}\di\xi.
\end{equation*}
We substitute (\ref{4.53}), (\ref{4.59}), (\ref{4.60}) into the last
identity,
\begin{align*}
Z(0,\b)=&-\b^2\sum\limits_{n=1}^{+\infty}\frac{1}{n\sqrt{4n^2-\b^2}}
\int\limits_0^{+\infty} \E^{-(2n+\sqrt{4n^2-\b^2})\xi_2} \di \xi_2
\\
=&-\b^2\sum\limits_{n=1}^{+\infty}
\frac{1}{n\sqrt{4n^2-\b^2}(2n+\sqrt{4n^2-\b^2})}
\end{align*}
that proves (\ref{4.58}).

The series in the definition of $\tht$ converges uniformly in $\b$,
and by the first Weierstrass theorem this function is holomorphic in
small $\b$. It is easy to see that
\begin{align*}
&\frac{1}{n\sqrt{4n^2-\b}(2n+\sqrt{4n^2-\b})}
=\frac{2n-\sqrt{4n^2-\b}}{\b n\sqrt{4n^2-\b}}
\\
&=\frac{1}{\b}\left(\frac{2}{\sqrt{4n^2-\b}} -\frac{1}{n}\right)
=\frac{1}{\b}\left(
\frac{1}{n\sqrt{1-\frac{\b}{4n^2}}}-\frac{1}{n}\right)=
\sum\limits_{j=1}^{+\infty} \frac{(2j-1)!!\b^{j-1}}{8^j n^{2j+1}
j!}.
\end{align*}
We substitute this identity into the definition of $\tht(\b)$,
\begin{equation*}
\tht(\b)=-\sum\limits_{n=1}^{+\infty} \sum\limits_{j=1}^{+\infty}
\frac{(2j-1)!!\b^{j-1}}{8^j n^{2j+1} j!}=
-\sum\limits_{j=1}^{+\infty}
\frac{(2j-1)!!\z(2j+1)\b^{j-1}}{8^j j!},
\end{equation*}
which yields (\ref{4.58a}). The proof is complete.
\end{proof}

We choose the boundary layer as
\begin{equation}\label{4.61}
\psi_\e^{bl}(\xi,\L)=\e\sqrt{\L}\cos\sqrt{\L}\pi
\big(X(\xi)+Z(\xi,\e\sqrt{\L})\big).
\end{equation}
It is clear that this function satisfies all the aforementioned
requirements for the boundary layer.

In accordance with Lemma~\ref{lm4.4}, the boundary layer has a
logarithmic singularity at $\xi=0$, and the sum of the external
expansion and the boundary layer does not satisfy the boundary
condition on $\gp_\e$ in (\ref{4.4}). This is the reason of
introducing the internal expansion. We construct it as depending on
$\vs:=\vs^{(1)}$ and employ the method of matching of the asymptotic
expansions. It follows from (\ref{4.2}), (\ref{1.17}) that
\begin{align}
&\psi_\e^{ex}(x,\mu)=\psi_\e^{ex}(0,\mu)+\frac{\p\psi_\e^{ex}}{\p
x_2}(0,\mu)x_2+\Odr(|x|^2),\quad x\to0, \label{4.25a}
\\
&\psi_\e^{ex}(0,\mu)=-\sin\sqrt{\L(\e,\mu)}\pi,\label{4.62}
\end{align}
where the asymptotics is uniform in $\L(\e,\mu)$. Using the
definition of $\vs=\xi\eta^{-1}$ and (\ref{1.6}), by (\ref{4.61}),
(\ref{4.57}), (\ref{2.12}) we obtain
\begin{align*}
\psi_\e^{bl}(\xi,\L)=&\sqrt{\L}\cos\sqrt{\L}\pi
\left(-\frac{1}{\mu}+\e(\ln|\vs|+\ln 2)-x_2\right)
\\
&+\e^3 \L^{3/2}\tht(\e^2\L)\cos\sqrt{\L}\pi
+\Odr(\e|\xi|^2\ln|\xi|),\quad\xi\to0,
\end{align*}
uniformly in $\e$ and $\L$. 
In view of (\ref{4.51}),
(\ref{4.25a}), (\ref{4.62}) we have
\begin{align*}
\psi_\e^{ex}(x,\L)+\psi_\e^{bl}(\xi,\L)=&
-\frac{\sqrt{\L}}{\mu}\cos\sqrt{\L}\pi- \sin\sqrt{\L}\pi + \e^3
\L^{3/2}\tht(\e^2\L)\cos\sqrt{\L}\pi
\\
&+\e\sqrt{\L}\cos\sqrt{\L}\pi (\ln|\z|+\ln 2)
 +\Odr\big(\e\eta^2|\z|^2(|\ln|\z||+|\ln\eta|)\big),
\end{align*}
as $x\to0$. Hence, in accordance with the method of matching of
asymptotic expansions we conclude that the internal expansion should
be as follows,
\begin{equation}\label{4.64}
\psi_\e^{in}(\vs,\L)=\psi_0^{in}(\z,\L,\e)+ \e
\psi_1^{in}(\z,\L,\e),
\end{equation}
where the coefficients should satisfy the asymptotics
\begin{align}
&
\begin{aligned}
\psi_0^{in}(\vs,\L,\e)= &-\frac{\sqrt{\L}}{\mu}\cos\sqrt{\L}\pi-
\sin\sqrt{\L}\pi
\\
&+ \e^3 \L_1^{3/2} \tht(\e^2\L)\cos\sqrt{\L}\pi+o(1),\quad
\vs\to\infty,
\end{aligned} \label{4.65}
\\
& \psi_1^{in}(\vs,\L)=\e\sqrt{\L}\cos\sqrt{\L}\pi (\ln|\z|+\ln
2)+o(1),\quad \vs\to\infty.\nonumber
\end{align}
We substitute (\ref{4.64}) into (\ref{4.4}) and pass to the
variables $\vs$. It yields the boundary value problems for
$\psi_i^{in}$,
\begin{equation}
\D_\vs \psi_i^{in}=0,\quad \vs_2>0, \qquad \psi_i^{in}=0,\quad
\vs\in\gp^{1},\qquad \frac{\p \psi_i^{in}}{\p\vs_2}=0,\quad
\vs\in\Gp^{1}. \label{4.67}
\end{equation}
For $i=0$ this problem has the only bounded solution which is
trivial,
\begin{equation}\label{4.68}
\psi_0^{in}=0.
\end{equation}
Thus, by (\ref{4.65}) we obtain the equation (\ref{4.69}) for
$\L(\e,\mu)$.

In view of the properties of the function $Y$ described in the third
section the function $\psi_1^{in}$ should be chosen as
\begin{equation}\label{4.70}
\psi_1^{in}(\z,\L,\e)=\e\sqrt{\L}\cos\sqrt{\L}\pi Y(\z).
\end{equation}
The formal constructing of $\l_1(0,\e)$ and $\po_\e$ is complete.

We proceed to the studying of the equation (\ref{4.69}). Since the
function $\tht$ is holomorphic by Lemma~\ref{lm4.4}, the function
\begin{equation*}
T(\e,\mu,\L):=\sqrt{\L}\cos\sqrt{\L}\pi+\mu\sin\sqrt{\L\pi}
-\e^3\mu\L^{3/2}\tht(\e^2 \L)\cos\sqrt{\L}\pi
\end{equation*}
is jointly holomorphic w.r.t. small $\e$, $\mu$, and $\L$ close to
$1/4$. Employing the formula (\ref{4.58a}), we continue $T$
analytically to complex values of $\e$, $\mu$, and $\L$.

As $\e=\mu=0$, the equation (\ref{4.69}) becomes
\begin{equation*}
\sqrt{\L}\cos\sqrt{\L}\pi=0,
\end{equation*}
and it has the root $\L=1/4$. It is clear that
\begin{equation*}
\frac{\p T}{\p\L}\left(0,0,\frac{1}{4}\right)\not=0.
\end{equation*}
Hence, by the inverse function theorem there exists the unique root
of the equation (\ref{4.69}). This root is jointly holomorphic in
$\e$ and $\mu$ and satisfies (\ref{4.50}). We represent this root as
\begin{equation}\label{4.71}
\L(\e,\mu)=\L_0(\mu)+\sum\limits_{j=1}^{+\infty}\e^j
\widetilde{K}_j(\mu),
\end{equation}
where $\widetilde{K}_j(\mu)$ are holomorphic in $\mu$ functions. We
choose the leading term in this series as $\L_1(\mu)$, since as
$\e=0$ the equation (\ref{4.69}) coincides with (\ref{1.17}).

We substitute (\ref{4.71}) and (\ref{4.58a}) into (\ref{4.69}) and
equate the coefficients at $\e^i$, $i=1,\ldots,8$. It implies the
equations for $\widetilde{K}_i$, $i=1,\ldots,8$. Solving these
equations, we obtain $\widetilde{K}_1=\widetilde{K}_2=0$ and
(\ref{1.20}).

Let us prove that $\widetilde{K}_{2j+1}(\mu)=\mu^2 K_{2j+1}(\mu)$,
$\widetilde{K}_{2j}(\mu)=\mu^3 K_{2j}(\mu)$, where $K_j(\mu)$ are
holomorphic in $\mu$ functions. It is sufficient to prove that
\begin{equation*}
\widetilde{K}_j(0)=\widetilde{K}'_j(0)=0,\quad
\widetilde{K}''_{2j}(0)=0.
\end{equation*}
We take $\mu=0$ in (\ref{4.69}) and (\ref{4.71}),
\begin{gather}
\sqrt{\L(0,\e)}\cos\sqrt{\L(0,\e)}\pi=0,
\\
\L(0,\e)=\frac{1}{4}.\label{4.72a}
\end{gather}
By (\ref{1.16}), (\ref{4.71}) it implies $\widetilde{K}_j(0)=0$. We
differentiate the equation (\ref{4.69}) w.r.t. $\mu$ and then we let
$\mu=0$. It implies the equation
\begin{align*}
&- \frac {1}{2} \frac{\pi\sqrt{\L(\e,0)}\sin \sqrt{\L(\e,0)}\pi-
\cos\sqrt{\L(\e,0)}\pi}{\sqrt{\L(\e,0)}}\frac{\p\L}{\p\mu}(\e,0)
\\
&- \e^3\L^{3/2}(\e,0)\tht(\e^2\L(\e,0))\cos\sqrt{\L(\e,0)}\pi+
\sin\sqrt{\L(\e,0)} \pi=0.
\end{align*}
We substitute here the identity (\ref{4.72a}) and arrive at the
equation
\begin{equation*}
-\frac{\pi}{2}\frac{\p\L}{\p\mu}(\e,0)+1=0,
\end{equation*}
which by (\ref{1.16}) implies
\begin{equation}\label{4.72b}
\frac{\p\L}{\p\mu}(\e,0)=\frac{2}{\pi}=\frac{\p
\L_1}{\p\mu}(0).
\end{equation}
These identities and (\ref{4.71}) yield $\widetilde{K}_j'(0)=0$.

We differentiate the equation (\ref{4.69}) twice w.r.t. $\mu$ and
then we let $\mu=0$ taking into account the identities
(\ref{4.72a}), (\ref{4.72b}), and (\ref{4.58a}),
\begin{gather*}
-\frac{4}{\pi}+\frac{\e^3}{2}\tht\Big(\frac{\e^2}{4}\Big)
-\frac{\pi}{2}\frac{\p^2\L}{\p\mu^2}(\e,0)=0,
\\
\frac{\p^2\L}{\p\mu^2}(\e,0)=\frac{1}{\pi^2}
\left(-8+\e^3\pi\tht\Big(\frac{\e^2}{4}\Big)\right)=-\frac{1}{\pi^2}
\left(8+\frac{\pi}{8}\sum\limits_{j=1}^{+\infty}
\frac{(2j-1)!!\z(2j+1)}{ 32^{j-1}\,j!} \e^{2j+1} \right).
\end{gather*}
Hence, $\widetilde{K}''_{2j}(0)=0$, $j\geqslant 1$.

To calculate all other coefficients of (\ref{1.23}) we substitute
this series and (\ref{4.58a}) into the equation (\ref{4.69}) and
then equate the coefficients of like powers of $\e$. It implies certain 
equations, which can be solved w.r.t. $K_i$. Since all the
coefficients in the expansion in $\e$ of $\tht$ and other terms in
the equation (\ref{4.69}) are real, the functions $K_i$ are real,
too. Hence, by (\ref{1.23}) the function $\L$ is real-valued for
real $\e$ and $\mu$.

We proceed to the justification of the asymptotics. Denote
\begin{equation}
\begin{aligned}
\Po_\e(x):=&\big(\psi_\e^{ex}(x,\L(\e,\mu))+\chi_1(x_2)
\psi_\e^{bl}(\xi,\L(\e,\mu)) \big)
\big(1-\chi_1(|\vs|\eta^{1/2})\big)
\\
&+\chi_1\big(|\vs|\eta^{1/2}\big)\psi_\e^{in}(\vs,\L(\e,\mu)).
\end{aligned}\label{5.30a}
\end{equation}
where, we remind, $\chi_1$ is the cut-off function introduced in the
third section.

\begin{lem}\label{lm4.5}
The function $\Po_\e\in C^\infty(\overline{\Om}_\e\setminus\{x:
x_1=\pm\e\eta,\, x_2=0\})$ belongs to the domain of $\Hpe(0)$,
satisfies the convergence
\begin{equation}\label{5.31a}
\left\|\Po_\e-\sin\frac{x_2-\pi}{2}\right\|_{L_2(\Pi)}=\Odr(\e^{1/2}\mu),
\quad \e\to+0,
\end{equation}
and solves the equation
\begin{equation}\label{4.75}
\big(\Hpe(0)-\L(\e,\mu)\big)\Po_\e=h_\e,
\end{equation}
where for the function $h_\e\in L_2(\Om_\e)$ an uniform in $\e$,
$\mu$, and $\eta$ estimate
\begin{equation}\label{4.76}
\|h_\e\|_{L_2(\Om_\e)}\leqslant C(\mu\E^{-2\e^{-1}}+\e\eta^{1/2})
\end{equation}
holds true.
\end{lem}

\begin{proof}
It follows from the definition of $\Po_\e$ that
\begin{equation}\label{4.76a}
\Po_\e\in C^\infty(\overline{\Om}_\e\setminus\{x: x_1=\pm\e\eta,\,
x_2=0\})\cap\Hoper^1(\Om_\e,\Gp_+).
\end{equation}
The boundary condition (\ref{4.5}), (\ref{4.62}), and (\ref{2.16})
for $Y$ yield those for $\Po_\e$,
\begin{equation}\label{4.77}
\Po_\e=0\quad \text{on}\quad \Gp_+\cup\gp_\e,\qquad
\frac{\p\Po_\e}{\p x_2}=0\quad\text{on}\quad \Gp_\e.
\end{equation}

Let us show that
\begin{equation}\label{4.77a}
-(\D_\xi+\L(\e,\mu))\Po_\e=h_\e,\quad x\in\Om_\e,
\end{equation}
where $h_\e\in L_2(\Om_\e)$ satisfies (\ref{4.76}). Employing the
equations (\ref{4.52}), (\ref{4.67}), we obtain
\begin{align}
&-(\D_\xi+\L)\Po_\e=h_\e,\quad
h_\e=-(h_\e^{(1)}+h_\e^{(2)}+h_\e^{(3)}), \label{4.78}
\\
&h_\e^{(1)}(x)=2\chi_1'(x_2)\frac{\p}{\p
x_2}\psi_\e^{bl}(\xi,\L(\e,\mu))+\chi_1''(x_2)
\psi_\e^{bl}(\xi,\L(\e,\mu)),\nonumber
\\
&h_\e^{(2)}(x)=\L(\e,\mu)\chi_1(|\vs|\eta^{1/2} )
\psi_\e^{in}(\vs,\L(\e,\mu)),\nonumber
\\
&h_\e^{(3)}(x)=2\nabla_x\chi_1(|\vs|\eta^{1/2})\cdot\nabla_x
\Po_\e^{mat}(x)+\Po_\e^{(mat)}(x)\D_x\chi_1(|\vs|\eta^{1/2}),\nonumber
\\
&\Po_\e^{(mat)}(x):=\psi_\e^{in}(\vs,\L(\e,\mu))-
\psi_\e^{ex}(x,\L(\e,\mu))-\psi_\e^{bl}(\xi,\L(\e,\mu)).
\end{align}
It is clear that $h_\e^{(i)}\in L_2(\Om_\e)$ that implies the same
for $h_\e$.

Due to (\ref{4.69}) the function $\psi_\e^{bl}$ can be rewritten as
follows,
\begin{align*}
\psi_\e^{bl}(\xi,\L(\e,\mu))=&\mu\big(\e^3 \L^{3/2}(\e,\mu)
\tht(\e^2\L(\e,\mu))\cos\sqrt{\L(\e,\mu)}\pi
\\
&-\sin\sqrt{\L(\e,\mu)}\pi\big)\big(X(\xi)
+Z(\xi,\e\sqrt{\L(\e,\mu)})\big).
\end{align*}
Thus,
\begin{align*}
h_\e^{(1)}(x)=&\mu\big(\e^3 \L^{3/2}(\e,\mu)
\tht(\e^2\L(\e,\mu))\cos\sqrt{\L(\e,\mu)}\pi
 -\sin\sqrt{\L(\e,\mu)}\pi\big)
\\
&\Big(2\chi_1'(x_2)\frac{\p}{\p x_2}+\chi_1''(x_2)\Big)
\big(X(\xi)+Z(\xi,\e\sqrt{\L(\e,\mu)})\big).
\end{align*}
The functions $\chi_1'(x_2)$, $\chi_1''(x_2)$ are non-zero only for
$1<x_2<\frac{3}{2}$ that corresponds to
$\e^{-1}<\xi_2<\frac{3}{2}\e^{-1}$. For such values of $\xi$ we can
use the series (\ref{4.53}), (\ref{4.59}) for $X$ and $Z$ which
converge in $C^k\big(\left\{\xi: \e^{-1}\leqslant \xi_2\leqslant
\frac{3}{2} \e^{-1},\,|\xi_1|\leqslant\frac{\pi}{2}\right\}\big)$.
It yields the exponential estimate for $h_\e^{(1)}$,
\begin{equation}\label{4.79}
\|h_\e^{(1)}\|_{L_2(\Om_\e)}\leqslant C\mu\E^{-2\e^{-1}},
\end{equation}
where the constant $C$ is independent of $\e$ and $\mu$.

Taking into account (\ref{4.68}), and replacing in (\ref{4.70}) the
factor $\sqrt{\L}\cos\sqrt{\L}\pi$ by $\mu\big(\e^3 \L^{3/2}(\e,\mu)
\tht(\e^2\L(\e,\mu))\cos\sqrt{\L(\e,\mu)}\pi
 -\sin\sqrt{\L(\e,\mu)}\pi\big)$ as we did it in (\ref{4.78}), we
estimate $h_\e^{(2)}$,
\begin{equation}\label{4.79a}
\begin{aligned}
\|h_\e^{(2)}\|_{L_2(\Om_\e)}^2\leqslant &
C\e^4\mu^2\eta^2\int\limits_{|\vs|<\eta^{-1/2}, \,\vs_2>0}
|Y(\vs)|^2\di\vs
\\
\leqslant & C\e^4\mu^2\eta|\ln^2\eta|\leqslant C\e^2\eta,
\end{aligned}
\end{equation}
where the constants $C$ are independent of $\e$, $\mu$, and $\eta$.

The asymptotics (\ref{2.12}), (\ref{4.57}), (\ref{2.18}), the
equation (\ref{4.69}), and the identities (\ref{4.2}), (\ref{4.61}),
(\ref{4.64}), (\ref{4.68}), (\ref{4.70}) imply the differentiable
asymptotics for $\Po_\e^{mat}$,
\begin{align*}
\Po_\e^{mat}(x)=&\e\sqrt{\L}\cos\sqrt{\L}\pi \big(\ln|\vs|+\ln
2+\Odr(|\vs|^{-2})\big)-\sin\sqrt{\L}(x_2-\pi)
\\
&-\e \sqrt{\L}\cos\sqrt{\L}\pi \big( \ln|\xi|+\ln
2+\e^2\L\tht(\e^2\L)-\xi_2+\Odr(|\xi|^2)\big)
\\
=&-\sin\sqrt{\L}(x_2-\pi) -\sin\sqrt{\L}\pi
+\sqrt{\L}x_2\cos\sqrt{\L}\pi
+\Odr\big(\e\mu(|\xi|^2+|\vs|^{-2})\big)
\\
=&\Odr\big(|x|^2+\e\mu(|\xi|^2+|\vs|^{-2})\big)
\end{align*}
uniformly in $\e$, $\mu$, and $\eta$ as
\begin{equation}\label{4.80}
 \e\eta^{1/2}<|x|<\frac{3}{2}\e\eta^{1/2},\quad x\in\Om_\e.
\end{equation}
Thus, for such $x$
\begin{align*}
&|\Po_\e^{mat}(x)|\leqslant C(\e(\e+\mu)\eta),
\\
&|\nabla_x \Po_\e^{mat}(x)|\leqslant C((\e+\mu)\eta^{1/2}),
\end{align*}
where the constants $C$ are independent of $x$, $\e$, $\mu$, and
$\eta$. Since the functions $\nabla_x\chi_1(|\vs|\eta^{1/2})$,
$\D_x\chi_1(|\vs|\eta^{1/2})$ are non-zero only for $x$ satisfying
(\ref{4.80}), the last inequalities for $\Po_\e^{mat}$ and $\nabla_x
\Po_\e^{mat}$ enable us to estimate $h_\e^{(3)}$,
\begin{equation*}
\|h_\e^{(3)}\|_{L_2(\Om_\e)}\leqslant C((\e+\mu)\eta^{1/2}),
\end{equation*}
where the constant $C$ is independent of $\e$, $\mu$, and $\eta$. We
sum the last estimate and (\ref{4.79}), (\ref{4.79a}),
\begin{equation*}
\|h_\e\|_{L_2(\Om_\e)}\leqslant C(\mu\E^{-2\e^{-1}}+\e\eta^{1/2}),
\end{equation*}
where the constant $C$ is independent of $\e$, $\mu$, and $\eta$.
This estimate imply (\ref{4.76}).

Due to the smoothness (\ref{4.76a}) of $\Po_\e$, the boundary value
conditions (\ref{4.77}), and the equation (\ref{4.77a}), the
function $\Po_\e$ is a generalized solution to the boundary value
problem (\ref{4.77a}), (\ref{4.77}). Hence, $\Po_\e$ belongs to the
domain of $\Hpe(0)$.

Let us prove the estimate (\ref{5.31a}). Completely as in the
estimating $h_\e$, we check that
\begin{equation*}
\|\chi_1(x_2) \psi_\e^{bl} \big(1-\chi_1(|\vs|\eta^{1/2})\big)
+\chi_1\big(|\vs|\eta^{1/2}\big)\psi_\e^{in} -\psi_\e^{ex}
\chi_1(|\vs|\eta^{1/2})\|_{L_2(\Om_\e)}=\Odr(\e^2\mu).
\end{equation*}
In view of (\ref{1.16}) and the definition (\ref{4.2}) of
$\psi_\e^{ex}$ the estimate
\begin{equation*}
\left\|\psi_\e^{ex}-\sin\frac{x_2-\pi}{2}\right\|_{L_2(\Pi)}=
\Odr(\e^{1/2}\mu)
\end{equation*}
 holds true. Two last estimates and the definition (\ref{5.30a}) of
$\Po_\e$ imply (\ref{5.31a}).
\end{proof}

We proceed to the estimating of the error terms. The core of these
estimates are Lemmas~12,~13 in \cite{VL}. We employ these results in
the form they were formulated in \cite[Ch. I\!I\!I, Sec. 1.1, Lm.
1.1]{OIS}. For the reader's convenience we provide this lemma below.
\begin{lem}\label{lm4.6}
Let $\mathcal{A}: H\to H$ be a continuous linear compact
self-adjoint operator in a Hilbert space $H$. Suppose that there
exist a real $M>0$ and a vector $u\in H$, such that $\|u\|_{H}=1$
and
\begin{equation*}
\|\mathcal{A}u-M u\|_{H}\leqslant \k,\quad \a=const>0.
\end{equation*}
Then there exists an eigenvalue $M_i$ of operator $\mathcal{A}$ such
that
\begin{equation*}
|M_i-\mu|\leqslant \k.
\end{equation*}
Moreover, for any $d>\k$ there exists a vector $\overline{u}$ such
that
\begin{equation*}
\|u-\overline{u}\|_{H}\leqslant 2\k d^{-1},\quad
\|\overline{u}\|_{H}=1,
\end{equation*}
and $\overline{u}$ is a linear combination of the eigenvectors of
the operator $\mathcal{A}$ corresponding to the eigenvalues of
$\mathcal{A}$ from the segment $[M-d,M+d]$.
\end{lem}

Since the operator $\Hpe(0)$ is non-negative and self-adjoint in
$L_2(\Om_\e)$ and satisfies (\ref{3.1}), the inverse
$\mathcal{A}:=\Hpe^{-1}(0)$ exists, is bounded and self-adjoint, and
satisfies the estimate
\begin{equation}\label{5.46}
\|\mathcal{A}\|\leqslant 4.
\end{equation}
The operator $\mathcal{A}$ is also bounded as that from
$L_2(\Om_\e)$ into $\H^1(\Om_\e)$ and in view of the compact
embedding of $\H^1(\Om_\e)$ in $L_2(\Om_\e)$ the operator
$\mathcal{A}$ is compact in $L_2(\Om_\e)$.

We rewrite the equation (\ref{4.75}) as follows,
\begin{equation*}
\L^{-1}(\e,\mu)\Po_\e=\mathcal{A}\Po_\e+\widetilde{h}_\e,\quad
\widetilde{h}_\e:=\L^{-1}(\e,\mu)\mathcal{A}h_\e.
\end{equation*}
By (\ref{4.50}), (\ref{1.16}), (\ref{5.46}), (\ref{4.76}) the
function $\widetilde{h}_\e$ satisfies the estimate
\begin{equation*}
\|\widetilde{h}_\e\|_{L_2(\Om_\e)}=\Odr(\mu\E^{-2\e^{-1}}+\e\eta^{1/2}).
\end{equation*}
Hence, by (\ref{5.31a})
\begin{equation*}
\|\widetilde{h}_\e\|_{L_2(\Om_\e)} \|\Po_\e\|_{L_2(\Om_\e)}^{-1}
=\Odr(\mu\e^{-1/2}\E^{-2\e^{-1}}+\e^{1/2}\eta^{1/2}).
\end{equation*}
Taking this estimate into account, we apply Lemma~\ref{lm4.6} with
\begin{equation}\label{5.53}
\begin{aligned}
&H=L_2(\Om_\e), && u=\frac{\Po_\e}{\|\Po_\e\|_{L_2(\Om_\e)}},
\\
&M=\L^{-1}(\e,\mu), &&
\k=\|\widetilde{h}_\e\|_{L_2(\Om_\e)}\|\Po_\e\|_{L_2(\Om_\e)}^{-1},
\end{aligned}
\end{equation}
and conclude that there exists an eigenvalue $\widetilde{M}(\e,\mu)$
of $\mathcal{A}$ satisfying the estimate
\begin{equation*}
|\widetilde{M}(\e,\mu)-\L^{-1}(\e,\mu)|=
\Odr(\mu\e^{-1/2}\E^{-2\e^{-1}}+\e^{1/2}\eta^{1/2}).
\end{equation*}
Thus, by (\ref{4.50}), (\ref{1.16})
\begin{align}
&|\widetilde{M}(\e,\mu)|\geqslant |\L^{-1}(\e,\mu)|-
\Odr(\mu\e^{-1/2}\E^{-2\e^{-1}}+\e^{1/2}\eta^{1/2})\geqslant 3,\quad
|\widetilde{M}^{-1}(\e,\mu)|\leqslant \frac{1}{3},\nonumber
\\
&
\begin{aligned}
|\widetilde{M}^{-1}(\e,\mu)-\L(\e,\mu)|=&
\Odr\big((\mu\e^{-1/2}\E^{-2\e^{-1}}+\e^{1/2}\eta^{1/2})
|\L(\e,\mu)||\widetilde{M}^{-1}(\e,\mu)|\big)
\\
=&\Odr(\mu\e^{-1/2}\E^{-2\e^{-1}}+\e^{1/2}\eta^{1/2}).
\end{aligned}\label{5.51}
\end{align}
The number $\widetilde{M}^{-1}(\e,\mu)$ is an eigenvalue of
$\Hpe(0)$. Due to (\ref{1.15}), (\ref{1.16}) there exists exactly
one eigenvalue of this operator satisfying (\ref{5.51}), and this
eigenvalue is $\l_1(0,\e)$. Thus,
\begin{equation}\label{5.52}
|\l_1(0,\e)-\L(\e,\mu)|=
\Odr(\mu\e^{-1/2}\E^{-2\e^{-1}}+\e^{1/2}\eta^{1/2})
\end{equation}
that proves (\ref{1.19}).

The asymptotics (\ref{1.15}), (\ref{1.16}), (\ref{4.50}),
(\ref{1.19}) imply that for $\e$ small enough the segment
$[\L(\e,\mu)-1,\L(\e,\mu)+1]$ contains exactly one eigenvalue of
$\Hpe$, which is $\l_1(0,\e)$. Bearing in mind this fact and
(\ref{4.76}), we apply Lemma~\ref{lm4.6} with $d=1$ and other
quantities given by (\ref{5.53}) and conclude that the normalized in
$L_2(\Om_\e)$ eigenfunction $\pho(x,\e)$ associated with
$\l_1(0,\e)$ satisfies the estimate
\begin{equation*}
\left\|\frac{\Po_\e}{\|\Po_\e\|_{L_2(\Om_\e)}}-\pho(\cdot,\e)
\right\|_{L_2(\Om_\e)}\leqslant
\frac{2\|h_\e\|_{L_2(\Om_\e)}}{\|\Po_\e\|_{L_2(\Om_\e)}}\leqslant
\frac{C\big(\mu\E^{-2\e^{-1}}+\e\eta^{1/2}\big)}{\|\Po_\e\|_{L_2(\Om_\e)}},
\end{equation*}
where the constant $C$ is independent of $\e$, $\mu$, and $\eta$.
Hence, for the eigenfunction
$\po(x,\e):=\|\Po_\e\|_{L_2(\Om_\e)}\pho(x,\e)$ associated with
$\l_1(0,\e)$ we have
\begin{equation}\label{5.54}
\|\po(\cdot,\e)-\Po_\e\|_{L_2(\Om_\e)}=\Odr
\big(\mu\E^{-2\e^{-1}}+\e\eta^{1/2}\big).
\end{equation}

Denote $\Pho_\e(x):=\Po_\e(x)-\po(x,\e)$. The equations (\ref{4.75})
and the eigenvalue equation for $\po(x,\e)$ imply the equation for
$\Pho_\e$,
\begin{equation*}
\Hpe(0)\Pho_\e=\l_1(0,\e)\Pho_\e+\big(\l_1(0,\e)-\L(\e,\mu)\big)
\Po_\e.
\end{equation*}
Hence, we can write the integral identity
\begin{equation*}
\|\nabla\Pho_\e\|_{L_2(\Om_\e)}^2=\l_1(0,\e)\|\Pho_\e\|_{L_2(\Om_\e)}^2+
\big(\l_1(0,\e)-\L(\e,\mu)\big)(\Po_\e,\Pho_\e)_{L_2(\Om_\e)}.
\end{equation*}
Thus, by (\ref{5.54}), (\ref{5.52}), (\ref{5.31a}), (\ref{1.19}),
(\ref{4.50}), (\ref{1.16})
\begin{align*}
\|\nabla\Pho_\e\|_{L_2(\Om_\e)}^2\leqslant &
\l_1(0,\e)\|\Pho_\e\|_{L_2(\Om_\e)}^2+
\big(\l_1(0,\e)-\L(\e,\mu)\big)(\Po_\e,\Pho_\e)_{L_2(\Om_\e)}
\\
\leqslant & \|\Pho_\e\|_{L_2(\Om_\e)}^2 +
|\l_1(0,\e)-\L(\e,\mu)|\|\Po_\e\|_{L_2(\Om_\e)}\|\Pho_\e\|_{L_2(\Om_\e)}
\\
\leqslant & C\big( \mu^2\E^{-4\e^{-1}}+\e^2\eta\big).
\end{align*}
The last estimate and (\ref{5.54}) prove the asymptotics
(\ref{1.25}). Theorem~\ref{th1.5} is proved.

\end{document}